\documentclass[12pt]{amsart}

\usepackage{amsrefs}
\usepackage{amssymb}
\usepackage{fancyhdr}
\usepackage{bbm}
\usepackage{color}

\usepackage{xcolor}
\usepackage{hyperref}
\hypersetup{
  colorlinks=true,
  linkcolor=blue,
  citecolor=red
}

\usepackage{graphicx}
\usepackage{xifthen}

\newtheorem{theorem}{Theorem}[section]
\newtheorem{lemma}[theorem]{Lemma}
\newtheorem{prop}[theorem]{Proposition}
\theoremstyle{definition}
\newtheorem{definition}[theorem]{Definition}
\newtheorem{example}[theorem]{Example}

\newtheorem{cor}[theorem]{Corollary}
\newtheorem{thm}[theorem]{Theorem}
\newtheorem{claim}[theorem]{Claim}
\newcommand{\define}[1]{{\bf #1}}

\theoremstyle{remark}
\newtheorem{remark}[theorem]{Remark}

\numberwithin{equation}{section}
\newcommand{\s}{K}
\newcommand{\core}{\mathrm{Core}}

\newcommand{\nO}[1][]{\ifthenelse{\isempty{#1}}{\core_\emptyset(\s)}{\core_\emptyset(#1)}}
\newcommand{\nI}[1][]{\ifthenelse{\equal{#1}{}}{\core_G(\s)}{\core_G(#1)}}

\newcommand{\freegroup}[1][]{\ifthenelse{\equal{#1}{}}{\mathbb{F}_2}{\mathbb{F}_{#1}}}

\newcommand{\Integer}{\mathbb{Z}}

\newcommand{\Z}{\Integer}
\newcommand{\N}{\mathbb{N}}

\newcommand{\R}{\mathbb{R}}

\newcommand{\edge}{\epsilon}

\newcommand{\salgebra}[1][]{\ifthenelse{\equal{#1}{}}{\mathcal{T}}{\mathcal{T}_{#1}}}
\newcommand{\unionAlgebra}[1][]{\ifthenelse{\equal{#1}{}}{\sigma\left(\bigcup_{n}\salgebra[C_n]\right)}{\sigma\left(\bigcup_{n}\salgebra[#1]\right)}}

\newcommand{\sinf}{\mathrm{Sym}_*(X)}
\newcommand{\fullsinf}{\mathrm{Sym}(X)}

\newcommand{\n}{N_G(\s)}
\newcommand{\percspace}{2^{\n \backslash G}}
\newcommand{\prob}[1]{Prob\left( #1 \right)}
\newcommand{\norm}[1]{\left\Vert #1 \right\Vert _{\s}}
\newcommand{\rwe}{h_{RW}\left( G,\mu\right)}
\newcommand{\rweq}[1]{h_{RW} \left(G/ #1,\bar{\mu}\right)}

\newcommand{\subg}{Sub_G}

\newcommand{\acts}{\curvearrowright}
\newcommand{\eps}{\varepsilon}

\newcommand{\pref}{\mathrm{pref}}
\newcommand{\shd}{\mathrm{shd}}
\newcommand{\Tail}{\mathcal{T}}
\newcommand{\set}[1]{\left\{ #1 \right\} }

\newcommand{\E}{\mathbb{E}}
\renewcommand{\Pr}{\mathbb{P}}

\begin{document}

\title[Entropy of Intersectional IRSs]{Furstenberg entropy of Intersectional Invariant Random Subgroups}


\author[Yair Hartman]{Yair Hartman}
\address{YH: Northwestern University, Evanston, IL}
\email{hartman@math.northwestern.edu}

\author[Ariel Yadin]{Ariel Yadin}
\address{AY: Ben-Gurion University of the Negev, Be'er Sheva ISRAEL}
\email{yadina@bgu.ac.il}

\thanks{Y.\ Hartman is supported by the Fulbright Post-Doctoral Scholar Program.
  A.\ Yadin is supported by the Israel Science Foundation (grant no.\ 1346/15). \\ 
  We would like to thank Lewis Bowen and Yair Glasner for helpful discussions.
}
\date{\today}

\subjclass[2010]{Primary }

\date{}

\dedicatory{}

\commby{}

\begin{abstract}
We study the {\em Furstenberg-entropy realization problem} for stationary actions.
It is shown that for finitely supported probability measures on free groups, 
any {\em a-priori} possible entropy value can be realized 
as the entropy of an ergodic stationary action. This generalizes results of Bowen.

The stationary actions we construct arise via invariant random subgroups (IRSs), 
based on ideas of Bowen and Kaimanovich.
We provide a general framework for constructing a continuum of ergodic IRSs for a discrete 
group under some algebraic conditions, which gives a continuum of entropy values.  
Our tools apply for example, for certain extensions of 
the group of finitely supported permutations and lamplighter groups, hence 
establishing full realization results for these groups.

For the free group, we construct the IRSs via a geometric construction of
subgroups, by describing their Schreier graphs.  
The analysis of the entropy of these spaces is obtained by studying the random walk
on the appropriate Schreier graphs.
\end{abstract}

\maketitle


\section{Introduction}

In this paper we study the {\em Furstenberg-entropy realization problem}, which we will describe below.
We propose a new generic construction of {\em invariant random subgroups} and we are able to analyse the Furstenberg-entropy
via the random walk and its properties.
We apply this construction to free groups and lamplighter groups, 
generalizing results of Bowen \cite{bowen2010random} and of the first named author
with Tamuz \cite{hartman2012furstenberg}. In addition, we apply it to the group
of finitely supported infinite 
permutations and to certain extensions of this group, 
establishing a full realization result for this class of groups.
Let us introduce the notions and results precisely.

Let $G$ be a countable discrete group and let $\mu$ be a probability measure on $G$.
We will always assume that $\mu$ is
generating and of finite entropy; that is, the support of 
$\mu$ generates $G$ as a semigroup and the Shannon entropy of $\mu$ is finite, 
$H(\mu) := -\sum_g \mu(g) \log \mu(g) <\infty$.

Suppose that $G$ acts measurably on a standard probability space $(X,\nu)$. 
This induces an action on the probability measure $\nu$ given by $g \nu(\cdot) = \nu(g^{-1} \cdot)$.
The action $G \acts (X,\nu)$ is called \define{$(G,\mu)$-stationary}
if $\sum_{g\in G} \mu(g)g\nu = \nu$. 
Thus, the measure $\nu$ is not necessarily invariant, 
but only ``invariant on average''. 
An important quantity associated to a stationary space is the Furstenberg-entropy, given by
$$ h_\mu(X,\nu) :=-\sum_{g\in G} \mu(g) \int_X \log \frac{dg^{-1}\nu}{d\nu}d\nu(x)$$

It is easy to see that $h_\mu(X,\nu)\geq 0$ and that equality holds if and only if 
the measure $\nu$ is invariant ({\em invariant} means $g \nu = \nu$ for all $g \in G$).
A classical result of Kaimanovich-Vershik 
\cite{kaimanovich1983random} asserts that the Furstenberg-entropy of any
$(G,\mu)$-stationary space is bounded above by the random walk entropy 
$$ \rwe :=\lim_{n\to\infty} \frac{1}{n}H(\mu^n) , $$ 
where $\mu^n$ denotes the 
$n$-th convolution power of $\mu$.
As a first step of classification of the possible 
$(G,\mu)$-stationary actions for a given $\mu$, one may consider the following definitions.
\begin{definition}
We say that $(G,\mu)$ has an \define{entropy gap} if there exists some $c>0$ such that 
whenever $h_\mu(X,\nu)<c$ for an ergodic $(G,\mu)$-stationary action $(X,\nu)$ then $h_\mu(X,\nu)=0$.
Otherwise we say that $(G,\mu)$ has \define{no-gap}.

We say that $(G,\mu)$ admits a \define{full realization} if any number in $[0,\rwe]$ can be realized as the Furstenberg-entropy of some ergodic $(G,\mu)$-stationary space.
\end{definition}

Let us remark that choosing only ergodic actions $G \acts (X,\nu)$ is important. 
Otherwise the definitions are non-interesting, since by taking convex combinations one 
can always realize any value.

The motivation for the gap definition comes from Nevo's result in \cite{nevo2003spectral} showing that any discrete group with property (T) admits an entropy gap, for any $\mu$. The converse however is not true (see {\em e.g.}\ Section 7.2 in \cite{bowen2014generic}) and it is interesting to describe the class of groups 
(and measures) with no-gap. 

Our main result is the following.
\begin{theorem}\label{thm:main-free}
Let $\mu$ be a finitely supported, generating probability measure on the free group
$\freegroup[r]$ on $2\leq r < \infty$ generators. 
Then $(\freegroup[r],\mu)$ admits a full realization.
\end{theorem}

As will be explained below (see Remark \ref{rem:finite supp}), 
the restriction to finitely supported measures $\mu$ 
is technical and it seems that the result should hold in a wider generality.
The case where $\mu$ is the simple random walk (uniform measure on the generators and their inverses) was done by Bowen \cite{bowen2010random}. We elaborate on the similarity and differences with Bowen's result below.

%

Our next result provides a solution for lamplighter groups.

\begin{theorem}\label{thm:main-lamplighter}
Let $G=\oplus_B L\rtimes B$ be a lamplighter group where $L$ and $B$ are some non trivial countable discrete groups.
Let $\mu$ be a generating probability measure with finite entropy on $G$ and
denote  by  $\bar{\mu}$ the projection of $\mu$ onto the quotient $B \cong G / (\oplus_B L \rtimes \{e\})$.

Then, whenever $(B,\bar{\mu})$ is Liouville, $(G,\mu)$ admits a full realization.
\end{theorem}

Here, {\em Liouville} is just the property that the random walk entropy is zero.
For example, if the base group $B$ is virtually nilpotent (equivalently, of polynomial growth in the finitely generated context),
then $B$ is Liouville for any measure $\mu$ (see Theorem \ref{fact:nilpotent-are-Liouville} below).

%

In \cite{hartman2012furstenberg}, the first named author and Tamuz prove that
a dense set of values can be realized, under an additional technical condition on the random walk.

In addition, our tools apply also to nilpotent extensions of
$\sinf$, the group of finitely supported permutations of an infinite countable set $X$. 
For simplicity, we state here the result for $\sinf$ itself, and refer the reader to Section \ref{sec:lamplighter} 
for the precise statement and proof of the general result.

\begin{theorem}\label{thm:main-sinf}
Let $\mu$ be a generating probability measure with finite entropy on $\sinf$. 
Then $(\sinf,\mu)$ admits a full realization.
\end{theorem}

Following the basic approach of Bowen in \cite{bowen2010random}, our realization results use only a specific type of stationary actions known as {\em Poisson bundles}
that are constructed out of invariant random subgroups. 
An \define{invariant random subgroup}, or \define{IRS}, of $G$, is random subgroup whose law is invariant 
to the natural $G$-action by conjugation.
IRSs serve as a stochastic generalization of a normal subgroup and arise naturally as the stabilizers of probability measure preserving actions.
In fact, any IRS is obtained in that way (see Proposition 14 in \cite{abert2014kesten} for the discrete case, and \cite{abert2012growth} for the general case).
Since the term IRS was coined in \cite{abert2011measurable}, they have featured in many works and inspired much research. 

In this paper we construct explicit families of 
IRSs for the free group, the lamplighter group and $\sinf$. For further discussion
about the structure of IRSs on these groups the reader is referred to the works of
Bowen~\cite{bowen2015invariant} for the free group, Bowen, Grigorchuk and Kravchenko~\cite{bowen2012invariantlamp} for lamplighter groups and Vershik~\cite{vershik2012totally}
for the $\sinf$ case.

Given an ergodic IRS of a group $G$, one can construct an ergodic stationary space which
is called a {\em Poisson bundle}. These were introduced by Kaimanovich in \cite{kaimanovich2005amenability} and were
further studied by Bowen in \cite{bowen2010random}. 
In particular, it was shown in \cite{bowen2010random} that the Furstenberg-entropy of a Poisson bundle constructed using an IRS, equals the random walk entropy of the random walk on the associated coset space (this will be made precise below). 
Hence our main results can be interpreted as realization results for the random walk entropies of coset spaces associated with ergodic IRSs.

Given a subgroup with an infinite conjugacy class,
we develop a general tool to construct a family of ergodic IRSs, that we call
\define{intersectional IRSs}. 
This family of IRSs is obtained by randomly intersecting subgroups from the conjugacy class (see Section \ref{sec:percolation-IRS}). 
Furthermore, we prove that under a certain 
condition related to the measure $\mu$ and the conjugacy class, 
the random walk entropy varies continuously along this family of IRSs (see Section~\ref{sec:entropy-is-cont}). 
In fact, in some cases we may find a specific conjugacy class 
that satisfies the above mentioned condition for any measure $\mu$.  An example 
of such, is a conjugacy class that we call ``locally co-nilpotent''. 
These are conjugacy classes such that that their associated Schreier graphs satisfy some 
geometric condition which we describe in Section \ref{scn:local}.

Using such families of IRSs we get our main tool for constructing a continuum of entropy values.
\begin{prop}\label{thm:no-gap}
Let $G$ be a countable discrete group. 
Assume that there exists a conjugacy class of subgroups which is locally co-nilpotent.
(See just before Corollary \ref{cor:single-subgroup} and 
Section \ref{scn:local} for the precise definition of locally co-nilpotent.)
Denote its normal core by $N \vartriangleleft G$.

Then, for any finite entropy, generating, probability measure $\mu$ on $G$, if $G/N$ has positive
$\mu$-random walk entropy, then $(G,\mu)$ has no-gap.

Furthermore, if the normal core $N$ is trivial, then $(G,\mu)$ admits a full realization.
\end{prop}

Note that this proposition reveals a structural condition on $G$ and its subgroups that allows to conclude realization results for many different measures $\mu$ at the same time.
Furthermore, these realization results do not depend on the description of the Furstenberg-Poisson boundary.

By constructing conjugacy classes in lamplighter groups and in $\sinf$ with the relevant properties, we obtain Theorems~\ref{thm:main-lamplighter} and~\ref{thm:main-sinf}.
However, the case of the free group is more complicated. 
In Section~\ref{sec:gluing-graphs} we show how to construct many subgroups of the
free group with a conjugacy class which is locally co-nilpotent, and such that $G/N$ has positive 
random walk entropy where $N$ is the normal core of the conjugacy class. This is enough to realize an interval of entropy values around $0$, showing in particular, that the free group admits no-gap for any $\mu$ with finite entropy. The fact that the free group admits no-gap for any finite first moment $\mu$ was
proved in~\cite{hartman2012furstenberg}.

\begin{thm}
\label{thm:free group no gap}
Let $\mu$ be a finite entropy generating probability measure on the free group
$\freegroup[r]$ on $2\leq r < \infty$ generators. 
Then $(\freegroup[r],\mu)$ has an interval of the form $[0,c]$ of realizable values.
\end{thm}

However, there is no single (self-normalizing) subgroup of the free group which has
a locally co-nilpotent conjugacy class and a trivial normal core (see Lemma~\ref{lem:amenability-restriction}).
The importance of the requirement that the normal core $N$ is trivial, is to ensure that 
the random walk entropy on $G/N$ is equal to the full random walk entropy.
Hence, to prove a full realization for the free group, we approximate this property. 
That is, we find a sequence of conjugacy classes which are all locally co-nilpotent, and such that 
their normal cores satisfy that the random walk entropy on the quotients tends to $\rwe$. 

Finding such a sequence is quite delicate, and proving it is a bit technical. 
The difficulty comes from the fact that entropy is semi-continuous, but in the ``wrong direction'' (see Theorem \ref{thm:Bowen} and the discussion after). 
It is for this analysis that we need to assume that $\mu$ is finitely supported, although it is plausible that this condition can be relaxed.
By that we mean that the very same conjugacy classes we construct in this paper might satisfy
that random walk entropy on the quotients by their normal cores tends to $\rwe$, even for infinitely supported $\mu$ as well (perhaps under some further regularity conditions, such as some moment conditions).

\subsection{Intersectional IRSs}


We now explain how to construct intersectional IRSs.

Let $K \leq G$ be a subgroup with an infinite conjugacy class $|K^G|=\infty$
(we use the notation $K^g = g^{-1} K g$ for conjugation). 
We may index this conjugacy class by the integers:  $K^G = (K_z)_{z \in \Z}$.  
It is readily checked that the $G$-action (by conjugation) on the conjugacy class $K^G$ is translated to an 
action by permutations on $\Z$.  Namely, for $g \in G , z \in \Z$, we define $g.z$ to be the unique integer 
such that $(K_z)^g = K_{g.z}$.  Thus, $G$ also acts on subsets $S \subset \Z$ element-wise.

Given a non-empty subset $S \subset \Z$, 
we may associate to $S$ a subgroup $\core_S(K)= \bigcap_{s \in S} K_s$.
If $S$ is chosen to be random, with a law invariant to the $G$-action, 
then the random subgroup $\core_S(K)$ is an IRS.  
Moreover, it is quite simple to get such a law:  for a parameter $p \in (0,1]$ let $S$ be 
$p$-percolation on $\Z$; that is, every element $z \in \Z$ belongs to $S$ with probability $p$ independently.

There are few issues with this construction.

\begin{itemize}
\item It is not immediately clear what should be associated to the ``empty intersection'' when $S = \emptyset$.
In Section~\ref{scn:empty-core} we give an apropriate definition, which follows from a construction
of an associated ``semi-norm'' on the group $G$. This will turn out
to be a normal subgroup of $G$ that we denote by $\nO$.

\item Note that $p=1$ corresponds 
to the non-random subgroup 
$\nI$, which is just the normal core of the conjugacy class $K^G$. 
This is also a normal subgroup,
and $\nI \lhd \nO$.

\item It seems very reasonable that as $p$ varies, certain quantities related to the IRS should vary continuously.
In Section~\ref{sec:entropy-is-cont} we provide a condition on the conjugacy class $K^G$ that guarantees that the Furstenberg-entropy is indeed continuous in $p$.
\end{itemize}

After establishing continuity of the Furstenberg-entropy in $p$, one wishes to show that 
the full interval $[0,\rwe]$ of possible entropy values is realized. As mentioned
above, in the free group we must use a sequence of subgroups $K_n$ (or conjugacy classes), such that
the random walk entropy of the quotient $G/\nI[\s_n]$ is large, and push this quotient entropy all the way up to 
the full random walk entropy $\rwe$.  

This leads to a study of $\nI[K_n]$, the normal cores of 
a sequence of subgroups $(K_n)_n$, with the objective 
of showing that $\nI[K_n]$ are becoming ``smaller'' in some sense. As will be discussed  below, 
a naive condition such as having $\nI[K_n]$ converge to the trivial group does not suffice.
What is required, is to approximate the Furstenberg-Poisson boundary of the random walk on the full group by the boundaries of the quotients. 
By that we mean that we need to find a sequence such that one can recover the point in
the boundary of the random walk on the free group, by observing the projections in the
boundaries of the quotients $G/\nI[K_n]$.

However, we construct the subgroups $K_n$ geometrically,
by specifying their associated Schreier graphs 
(and not by specifying the algebraic properties of the subgroup). Hence the structure of $\nI[K_n]$ and in 
particular the Furstenberg-Poisson boundary of $G/\nI[K_n]$ is somewhat mysterious.  
A further geometric argument using the random walk on the aforementioned Schreier graph
(specifically transience of this random walk) allows us to ``sandwich'' the entropy of the $G/\nI[K_n]$  and thus show that these entropies tend to the correct limiting value.
This is done in Section~\ref{sec:free-group}.

It may be interesting for some readers to note that in the simple random walk case 
considered by Bowen \cite{bowen2010random}, he also shows a full realization by growing intervals 
of possible entropy values.  However, Bowen obtains the intervals of realization differently than us.
In fact, he proves that the entropy is continuous when restricting to a specific family
of subgroups which are ``tree-like''. This is a global property of the Schreier graph that enables him to analyze the random walk entropy of the simple random walk and to prove a continuity result. Next, within this class 
of tree-like subgroups he finds paths of IRSs obtaining realization of intervals 
of the form $[\eps_n, \rwe ]$ for $\eps_n \to 0$ (``top down'').  
While it is enough for a full realization for the simple random walk, the condition of
being {\em tree-like} is quite restrictive, and difficult to apply to other measures $\mu$ 
on the free group.

In contrast, based on Proposition \ref{thm:no-gap}, our results provide 
realization on intervals of the form $[0, \rwe - \eps_n]$ for $\eps_n \to 0$ (``bottom up''). In order
to ``push'' the entropy to the top we also use a geometric condition, related to trees. 
However, this condition is more ``local'' in nature, and 
considers only finite balls in the Scherier graph. In particular, we show in
Section \ref{scn:full real} that it is easy to construct many subgroups that 
satisfy our conditions.  Also, it enables us to work with any finitely supported measure.

\subsection{Related results}
The first realization result is due to Nevo-Zimmer~\cite{nevo2000rigidity}, where they prove that connected semi-simple Lie groups with finite center and $\R\mbox{-rank}$ at least $2$ admit only finitely many entropy values under some mixing condition on the action.
They also show that this fails for lower rank groups, such as $PSL_2(\R)$, by finding infinitely many entropy values. However, no other guarantees are given regarding these values.

The breakthrough of the realization results is the mentioned result of Bowen~\cite{bowen2010random}, where he makes a use of IRSs for realization purposes. He proves a full realization for simple random walks on free groups $\freegroup[r]$ where $2\le r <\infty$.

In~\cite{hartman2012furstenberg}, the first named author and Tamuz prove a dense realization for lamplighter groups with Liouville base. Furthermore, it is shown that
virtually free groups, such as $SL_2(\mathbb{Z})$, admit no-gap.

In a recent paper~\cite{burton2016weak}, Burton, Lupini and Tamuz prove that the 
Furstenberg-entropy is an invariant of weak equivalence of stationary actions.


\section{Preliminaries}
\subsection{Stationary spaces}
Throughout the paper we assume that $G$ is a discrete countable group
and that $\mu$ is a generating probability measure on $G$.

A measurable action $G \acts (X,\nu)$ where $(X,\nu)$ is a standard probability spaces is called \define{$(G,\mu)$-stationary} if $\sum_{g\in G} \mu(g)g\nu = \nu$. 
The basic example of non-invariant $(G,\mu)$-stationary action is the action on the {\em Furstenberg-Poisson boundary}, $\Pi(G,\mu)$ (see {\em e.g.}\ \cite{furstenberg1973boundary,furstenberg1971random,furstenberg2009stationary,
furman2002random} and references therein 
for more on this topic,
which has an extreme body of literature associated to it). 

The Kakutani fixed point theorem implies that whenever $G$ acts continuously on 
a compact space $X$, there is always some $\mu$-stationary measure on $X$. However, 
there are not many explicit descriptions of stationary measures. 
One way to construct stationary actions is via the {\em Poisson bundle} constructed from an IRS:
Denote by $\subg$ the space of all subgroups of $G$ (recall that $G$ is discrete).
This space admits a topology, induced by the product topology on subsets of $G$, known as the Chabauty topology. Under this topology, $\subg$ is a
compact space, with a natural right action by conjugation ($K^g=g^{-1}Kg$ where $K \in \subg$), 
which is continuous. 
An \define{Invariant Random Subgroup (IRS)} is a Borel probability measure $\lambda$ on $\subg$ which 
is invariant under this conjugation action. It is useful to think of an IRS also as a random subgroup 
with a conjugation invariant law $\lambda$. For more
background regarding IRSs we refer the interested reader to \cite{abert2011measurable,abert2014kesten}.
 
Fix a probability measure 
$\mu$ on $G$. Given an IRS $\lambda$ of $G$, one can consider the {\em Poisson bundle} over
$\lambda$, denoted by $B_\mu(\lambda)$. This is a probability space with a $(G,\mu)$ stationary action, 
and when $\lambda$ is an ergodic IRS, the action on $B_\mu(\lambda)$ is ergodic.

Poisson bundles were introduced in a greater generality by Kaimanovich in 
\cite{kaimanovich2005amenability}. In \cite{bowen2010random} Bowen relates Poisson bundles 
to IRSs. Although the stationary spaces that we construct in this paper are all Poisson bundles, we do not elaborate regarding this correspondence, as this will not be important for understanding our results. 
Rather, we refer the interested reader to \cite{hartman2013stabilizer,bowen2010random,hartman2012furstenberg,kaimanovich2005amenability}, and mention here the important features of Poisson bundles that are relevant to our context.
To describe these features, we now turn to describe Schreier graphs. 

%
%

\subsection{Schreier graphs}\label{sec:schreier-graphs}
Since we will use Schreier graphs only in the context of the free group, we assume in this section that 
$G$ is a free group. To keep the presentation simple, we consider
$G=\freegroup$ the free group on two generators, generated by $a$ and $b$.
The generalization to higher rank is straightforward.

Let $\s \in \subg$. The \define{Schreier graph} associated with $\s$ is a rooted oriented graph, 
with edges labeled by $a$ and $b$, which is defined as follows. 
The set of vertices is the coset space $\s \backslash G$ and the set of edges is 
$\{ (\s g,\s g a) , (\s g,\s g b) \ : \ \s g \in \s \backslash G \}$. 
Edges of the type $(\s g , \s g a)$ are labeled by $a$ and
of the type $(\s g , \s g b)$ are labeled by $b$. 
It may be the case that there are multiple edges between vertices,
as well as self loops. The root of the graph is defined to be the trivial coset $\s$.

From this point on, it is important that we are working with the free group, and not
a general group which is generated by two elements.
Consider an abstract (connected) rooted graph with oriented edges.
Suppose that every vertex has exactly two incoming edges labeled by $a$ and $b$, and two outgoing edges labeled by $a$ and $b$.
We may then consider this graph as a Schreier 
graph of the free group. 

Indeed, given such a graph $\Gamma$ one can recover the subgroup $\s$: 
Given any element $g=w_1\cdots w_n$, of the free group (where $w_1\cdots w_n$ is a reduced word in $\{a,a^{-1},b,b^{-1}\}$) and a vertex $v$ in the graph, 
one can ``apply the element $g$ to the vertex $v$'' to get another vertex denoted $v.g$, as follows:
\begin{itemize}
\item Given $w_i \in \{a,b\}$ move from $v$ to the adjacent vertex $u=v.w_i$ such that the oriented edge
$(v,u)$ is labeled by $w_i$.  

\item Given $w_i \in \{a^{-1} , b^{-1} \}$ move from $v$ to the adjacent vertex $u=v.w_i$ such that the 
oriented edge $(u,v)$ is labeled $w_{i}^{-1}$ (that is, follow the label $w_i$ using the reverse orientation).

\item Given $g = w_1 \cdots w_n$ written as a reduced word in the generators, 
apply the generators one at a time, to get $v.g = (\cdots((v.w_1).w_2)\ldots).w_n$.
\end{itemize}
%
Thus, we arrive at a right action of $G$ on the vertex set. 
Note that this action is not by graph automorphisms, 
rather only by permutations of the vertex set. We say that \define{$g$ preserves $v$} if $v.g=v$.

To recover the subgroup $\s\leq \freegroup$ from a graph $\Gamma$ as above, 
let $\s$ be the set of all elements of the
free group that preserve the root vertex. Moreover, changing the root - that is, recovering 
the graph which is obtained be taking the same labeled graph but when considering 
a different vertex as the root - yields a conjugation of $\s$ as the corresponding subgroup. 
Namely, if the original
root was $v$ with a corresponding subgroup $\s$, 
and if the new root is $v.g$ for some $g$, then one gets the corresponding subgroup $\s^g$.
Hence when ignoring the root, one can think of the Schreier graph as representing the conjugacy class $\s^G$. In particular, if $\s$ is the subgroup associated with a vertex $v$, and $g$ is such that when rooting the Schreier graph in $v$ and in $v.g$ we get two isomorphic labeled graphs, 
then $g$ belongs to the normalizer $\n = \{ g \in G \ : \ K^g = K \}$.


The subgroups of the free group that we will construct will be described by their Schreier graph. 
In the following section we explain the connection between Schreier graphs and Poisson bundles,
which is why using Shcreier graphs to construct subgroups is more adapted to our purpose.

\subsection{Random walks on Schreier graphs}
Fix a generating probability measure $\mu$ on a group $G$. 
Let $(X_t)_{t=1}^\infty$ be i.i.d.\ random variables taking values in $G$ with law
$\mu$, and let $Z_t=X_1 \cdots X_t$ be the position of the $\mu$-random walk at time $t$.

Let $\s$ be a subgroup. Although this discussion holds for any group, since we
defined Schreier graphs in the context of the free group, assume that $G=\freegroup$. 

The Markov chain $(\s Z_t )_t$ has a natural representation on the Schreier graph
associated with $\s$, by just considering the position on the graph at time $t$ to be $v.Z_t$ where $v$ is a root associated with $K$.
(This is the importance of considering \textit{left}-Schreier graphs and \textit{right}-random walks.)
This is a general correspondence, not special to the free group.

The boundary of the Markov chain $KZ_t$ where $K$ is fixed, is the boundary of the random walk on the 
graph that starts in $v$, namely, the collection of tail events of this Markov chain. The Poisson bundle is a bundle of such boundaries, or alternatively, it is the boundary
of the Markov chain $KZ_t$ where $K$ is a random subgroup distributed according to $\lambda$. This perspective was our intuition behind the construction in Section~\ref{sec:free-group}. 

\subsection{Random walk entropy}
Recall that we assume throughout that all random walk measures $\mu$ have finite
entropy, that is $H(Z_1)=-\sum_g \Pr [Z_1=g]\log(\Pr [Z_1=g]) < \infty$ where $Z_t$
as usual is the position of the $\mu$-random walk at time $t$. 
The sequence of numbers $H(Z_t)$ is a sub-additive sequence so the following
limit exists and finite.
\begin{definition}
The \define{random walk entropy} of $(G,\mu)$ is given by 
$$\rwe = \lim_{t\to\infty}  \frac{1}{t}H(Z_t).$$
\end{definition}

The \define{Furstenberg-entropy} of a $(G,\mu)$-stationary space, $(X,\nu)$ is given by
$$ h_\mu(X,\nu)=-\sum_{g\in G} \mu(g) \int_X \log \frac{dg^{-1}\nu}{d\nu}d\nu(x)$$

By Jensen's inequality, $h_\mu(X,\nu) \ge 0$ and equality holds if and only if the action
is a measure preserving action.
Furthermore,  Kaimanovich and Vershik proved in~\cite{kaimanovich1983random} that $h_\mu(X,\nu)\le \rwe$ for any stationary action $(X,\nu)$, and equality holds for the Furstenberg-Poisson boundary, namely
\begin{align}\label{eq:random-walk-entropy}
h_\mu(\Pi(G,\mu))=\rwe,
\end{align}
where $\Pi(G,\mu)$ denotes the Furstenberg-Poisson boundary of $(G,\mu)$.

If $\rwe=0$ the pair $(G,\mu)$ is called \define{Liouville}, which is equivalent to 
triviality of the Furstenberg-Poisson boundary.

The \define{Furstenberg-entropy realization problem} is to find the exact values in the interval
$[0,\rwe]$ which are realized as the Furstenberg entropy of ergodic stationary actions.

In a similar way to the Kaimanovich-Vershik formula \eqref{eq:random-walk-entropy}, Bowen 
proves the following:
\begin{theorem}[Bowen, \cite{bowen2010random}] \label{thm:Bowen}
Let $\lambda$ be an IRS of $G$. Then 
\begin{align}\label{eq:the-entropy-is-semi-continuous}
h_\mu (B_\mu(\lambda)) = \lim_{t\to\infty} \frac{1}{t}\int_{\subg} H(\s Z_t)d\lambda(\s)=\inf_{t }  \frac{1}{t}\int_{\subg} H(\s Z_t)d\lambda(\s)
\end{align}
Moreover, the $(G,\mu)$-stationary space is ergodic if (and only if) the IRS
$\lambda$ is ergodic.
\end{theorem}

This reduces the problem of finding $(G,\mu)$-stationary actions with many Furstenberg entropy values in $[0,\rwe]$, 
to finding IRSs $\lambda$ with many different random walk entropy values.

Regarding continuity,
recall that the space $\subg$ is equipped with the Chabauty topology. 
For finitely generated groups this topology is induced by the natural metric for which two subgroups
are close if they share the exact same elements from a large ball in $G$.
One concludes that the Shannon entropy, $H$, as a function on $\subg$ is a continuous function.
The topology on $\subg$ induces a weak* topology on, 
$\prob{\subg}$, the space of IRSs of $G$. 
By the continuity of $H$ and the second equality in \eqref{eq:the-entropy-is-semi-continuous} we get that the Furstenberg-entropy of Poisson bundles of the form 
$B_\mu(\lambda)$ is an upper semi-continuous function of the IRS. Since the entropy is non-negative we conclude:
\begin{cor}[Bowen, \cite{bowen2010random}]
\label{cor:entropy_semi_cont}
If $\lambda_n \to \lambda$ in the weak* topology on $\prob{\subg}$, such that $h_\mu(B_\mu(\lambda))=0$ then  $h_\mu(B_\mu(\lambda_n))\to 0$.
\end{cor}

However, it is important to notice that the Furstenberg-entropy is far from being a continuous function of the IRS. 
Indeed, consider the free group, which is residually finite. 
There exists a sequence of normal subgroups $N_n$, 
such that $N_n \backslash \freegroup$ is finite for any $n$, which 
can be chosen so that, when considered as IRSs $\lambda_n=\delta_{N_n}$,
we have $\lambda_n \to \delta_{\{e\}}$.  
In that case $h_\mu(B_\mu(\lambda_n))=0$ for all $n$ (because $H(N_n Z_t)$ is bounded in $t$), 
but $B_\mu(\delta_{\{e\}})$ is the Furstenberg-Poisson boundary of the free group,
and its Furstenberg entropy is positive for any $\mu$ (follows for example from the
non-amenability of the free group). 

We want to point out that this lack of continuity makes it delicate
to find conditions on quotients of the free group to have large entropy. 
In Section \ref{sec:free-group} 
we will deal with this problem by introducing a tool to prove that the entropy 
does convergence along a sequence of subgroups which satisfy some geometric conditions.

The free group is also residually nilpotent, so one can replace the finite quotient from
the discussion above by infinite nilpotent quotients. We arrive at a similar discontinuity of the entropy function, due to the following classical result. As we will use it in the sequel,  we quote it here for future reference. The original case of Abelian groups is due to Choquet-Deny \cite{choquet1960lequation} and the nilpotent case appears in \cite{raugi2004general}.

\begin{theorem}[Choquet-Deny-Raugi]
\label{fact:nilpotent-are-Liouville}
Let $G$ be a nilpotent group.  
For any probability measure $\mu$ on $G$ with finite entropy, the pair $(G,\mu)$ is Liouville ({\em i.e.}\ $\rwe = 0$).
\end{theorem}

We mention  that the triviality of the Furstenberg-Poisson boundary holds for general $\mu$ without any entropy assumption, 
although in this paper we consider only $\mu$ with finite entropy. 
The proof of the Choquet-Deny Theorem 
follows from the fact that the center of the group always acts trivially on the Furstenberg-Poisson 
boundary and hence Abelian groups are always Liouville. 
An induction argument extends this to nilpotent groups as well.

\section{Intersectional IRSs}\label{sec:percolation-IRS}
In this section we show how to construct a family of IRSs, given an infinite conjugacy class.

Let $G$ be a countable discrete group, and let $\s \le G$ be a subgroup with infinitely many different conjugates 
$|\s^G|  = \infty$. Equivalently, the normalizer $N_G(\s) : = \{ g \in G \ : \ K^g = K \}$ 
is of infinite index in $G$.

Recall that the $G$ action on $K^G$
is the right action $K^g=g^{-1}Kg$. 
Since $\s^{ng}=g^{-1} n^{-1} \s ng = g^{-1} \s g = \s^{g}$ for any
$n\in \n$, the $\s$-conjugation depends only on the coset $\n g$.  
That is, conjugation $K^\theta$ for a right coset $\theta \in \n \backslash G$ is well defined.

We use $\percspace$ to denote the family of all subsets of $\n \backslash G$.  
This is identified canonically with $\{ 0,1\}^{\n \backslash G}$.
Note that $G$ acts from the right on subsets in $\percspace$, via
$\Theta .g = \{ \theta g \ : \ \theta \in \Theta \}$.
Given a non-empty subset $\Theta \subset \n \backslash G$ we define the subgroup 
$$\core_\Theta(\s)=\bigcap_{\theta \in \Theta}\s^{\theta} .$$

\begin{claim}\label{claim-definition-of-path}
The map $\varphi:\percspace \to \subg$ defined by $\varphi(\Theta) = \core_\Theta(K)$ is $G$-equivariant.

In particular, if we denote by $\lambda_{p,K}$ the $\varphi$-push-forward of the 
Bernoulli-$p$ product measure on $\percspace$, 
then $\lambda_{p,K}$ is an ergodic IRS for any $p\in(0,1)$.
\end{claim}

\begin{proof}
Note that
$$
(\core_\Theta(K))^g=\Big(\bigcap_{\theta \in \Theta}K^{\theta } \Big)^{g}=
\bigcap_{\theta \in \Theta}(K^{\theta})^{g}=\bigcap_{\theta'\in \Theta.g}K^{\theta'}=\core_{\Theta.g}(K)
$$
It follows that we can push forward any $G$-invariant measure on $\percspace$ to get IRSs.
For any $p\in(0,1)$ consider the Bernoulli-$p$ product measure on 
$\percspace \cong \{0,1\}^{\n \backslash G}$, namely, each
element in $\n\backslash G$ is chosen to be in $\Theta$ independently with probability $p$. These measures 
are clearly ergodic invariant measures. It follows that the push-forward measures $\lambda_{p,K}$ are ergodic invariant measures on $\subg$.
\end{proof}


We continue to determine the limits of $\lambda_{p,K}$ as $p \to 0$ and $p \to 1$.
When the subgroup $K$ is clear from the context, we write $\lambda_p = \lambda_{p,K}$.

\begin{lemma} \label{lem:nO}
Given $\s \le G$ with $|\s^G|=\infty$, there exist two normal subgroups, 
$\nO,\nI \triangleleft G$ such that $\lambda_{p,K} \xrightarrow{p\to 0} \delta_{\nO}$  and 
$\lambda_{p,K} \xrightarrow{p\to 1} \delta_{\nI}$, 
where the convergence is in the weak* topology on $\prob{\subg}$.
\end{lemma}

While the definition of $\nI$ is apparent, 
it is less obvious how to define $\nO$.
We now give an intrinsic description of the normal subgroup $\nO$.

\subsection{The subgroup $\nO$}\label{scn:empty-core}

For any element $g\in G$, let $\Omega_g\subset \n \backslash G $ be the set of all cosets, such that $g$ belongs to the corresponding conjugation: $\Omega_g=\{ \theta \in \n \backslash G \ | \ g\in K^\theta\} $. 
For example, for $g\in \s$ we have that $\n \in \Omega_g$.

We observe that: 
\begin{enumerate}
\item $\Omega_{g}\cap\Omega_{h}\subset\Omega_{gh}$ 

\item $\Omega_{g^{-1}}=\Omega_{g}$ 

\item $\Omega_{g^{\gamma}}=\Omega_{g}\gamma^{-1}$ 

\end{enumerate}
The first two follow from the fact that $\s^\theta$ is a group and the third since $\gamma \s^\theta \gamma^{-1} = \s^{\theta \gamma^{-1}}$.

Let $\mho_g$ be the complement of $\Omega_g$ in $\n \backslash G$.
That is, $\mho_g = \{ \theta \in \n \backslash G \ : \ \theta \not\in \Omega_g \}$.
Define $\norm{g} = |\mho_g |$.
The properties above show that 
\begin{enumerate}
\item $\norm{gh}\le\norm{g}+\norm{h}$ 
\item $\norm{g^{-1}}=\norm{g}$
\item $\norm{g^{\gamma}}=\norm{g}$
\end{enumerate}

Two normal subgroups of $G$ naturally arise: the subgroup of all elements with 
zero norm and the subgroup of all elements with finite norm.  
By (1) and (2) above, these are indeed subgroups and normality follows from (3).

The first subgroup is, by definition, the normal core $\nI =\{ g \ | \ \norm{g}=0\}$.
We claim that the second subgroup is the appropriate definition for $\nO$. 
That is, we define $\nO  = \{ g \ | \ \norm{g} <\infty\}$ and 
prove that $\lambda_{p,K} \to \delta_{\nO}$ as $p\to 0$.

\begin{proof}[Proof of Lemma~\ref{lem:nO}]
By definition, to show that $\lambda_p \xrightarrow{p\to0} \delta_{\nO}$ we need to show that 
for any $g\in \nO$, we have $\lambda_p(\{F\in\subg \ | \ g\in F\} ) \xrightarrow{p\to0} 1$ and for any
$g \notin \nO$, we have $\lambda_p(\{F\in\subg \ | \ g\in F\} )\xrightarrow{p\to0} 0$.

Note that for any $\Theta \neq \emptyset$, by the definition of $\core_\Theta(\s)$, we have that $g\in \core_\Theta(\s)$ if and only if $\Theta\cap \mho_g=\emptyset$.

Now consider $\Theta$ to be a random subset according to the Bernoulli-$p$ product measure on $\percspace$. 
By the above, we have equality of the events
$\{ g\in \core_\Theta(K)\} = \{\mho_g \cap \Theta =\emptyset\}$. 
The latter (and hence also the former) has 
probability $\left(1-p\right)^{|\mho_g|}=(1-p)^{\norm{g}}$.
This includes the case where $|\mho_g|= \norm{g} = \infty$ 
(whence we have $0$ probability for $\{ \mho_g \cap \Theta = \emptyset \}$).
We conclude that for any $g \in G$,
$$ \lambda_p(\{F\in\subg \ | \ g\in F\}) = (1-p)^{\norm{g}}
=
\begin{cases}
0 & \textrm{ if } \norm{g} = \infty , \\
\to 1 \textrm{ as } p \to 0 & \textrm{ if } \norm{g} < \infty .
\end{cases}
$$
It follows that $\lambda_p \to \delta_{\nO}$ as $p\to 0$.

Finally, we show that $\lambda_p \to \delta_{\nI}$ as $p\to 1$, where we define
$\nI = \bigcap_g K^g = \core_{\n \backslash G}(K)$.
Clearly, for any $p\in(0,1)$, we have that $\nI \subset F$ for $\lambda_p$-almost every $F$, 
since $\nI \leq \core_\Theta(\s)$ for any non-empty subset $\Theta$.

On the other hand, fix some $g \not\in \nI$. By definition, there exists some $\theta \in \n \backslash G$ such that $g\not\in \s^\theta$. Whenever $\Theta$ is such that $\theta\in \Theta$, we have that $g\not\in \core_\Theta(K)$.  
The probability that $\theta\not\in \Theta$ is $(1-p)$.
Hence, for any $g$ with $\norm{g} > 0$, 
$$\lambda_p(\{F\in\subg \  | \ g \in F \})\leq (1-p) \xrightarrow{p\to 1} 0.$$
We conclude that $\lambda_p \to \delta_{\nI}$ as $p\to 1$.
\end{proof}

\subsection{Continuity of the entropy along intersectional IRSs}
\label{sec:entropy-is-cont}

Lemma \ref{lem:nO} above shows that the IRSs $\lambda_{p,K}$ interpolate between the two Dirac measures concentrated on the normal subgroups $\nO$ and $\nI$. In this section we provide a condition
under which this interpolation holds for the random walk entropy as well.

Fix a random walk $\mu$ on $G$ with
finite entropy. 
Let $h:[0,1]\to \mathbb{R}_+$ be the entropy map defined by 
$ h(p) = h_\mu(B_\mu(\lambda_p)) $ when $p \in (0,1)$
and $h(0), h(1)$ defined as the $\mu$-random walk entropies on the quotient groups 
$G/\nO$ and $G/\nI$ respectively.
In particular, the condition that $h(0)=0$ is equivalent to the Liouville property holding for
the projected $\mu$-random walk on $G/\nO$.

\begin{prop}\label{prop:continuity}
If $h(0)=0$ then $h$ is continuous.
\end{prop}


\begin{proof}
Let 
$$ H_n(p)= \int_{\subg} H(F Z_n) d\lambda_p(F) . $$ 
So $h(p)=\lim_{n\to\infty} \frac{1}{n}H_n(p)$.
First, we use a standard coupling argument to prove the for $q = p+\eps$,
\begin{equation}
\label{eq:bound-for-the-entropy}
0 \leq H_n(q)-H_n(p) \leq H_n(\eps) 
\end{equation}

Indeed, let $(U_\theta)_{\theta \in \n \backslash G}$ be i.i.d.\ uniform-$[0,1]$ random variables,
one for each coset of $\n$.  We use $\Pr, \E$ to denote the probability measure and expectation 
with respect to these random variables.
Define three subsets
$$ \Theta_p : = \{ \theta \ : \ U_\theta \leq p \} \qquad \Theta_\eps : = \{ \theta \ : \ p < U_\theta \leq q \} $$
and $\Theta_q : = \Theta_p \cup \Theta_\eps$.
We shorthand $C_x : = \core_{\Theta_x} (\s)$ for $x \in \{p,\eps,q\}$.
Note that for $x \in \{p,\eps,q\}$ the law of $C_x$ is $\lambda_x$.
Thus, 
$H_n(x) = \E [H( C_x Z_n)] $
(to be clear, the entropy is with respect to 
the random walk $Z_n$, and expectation is with respect to the random subgroup $C_x$).

Now, by definition
$$ C_q = \bigcap_{\theta \in \Theta_q} \s^\theta = 
\bigcap_{\theta \in \Theta_p} \s^\theta \cap \bigcap_{\theta \in \Theta_\eps} \s^\theta = C_p \cap C_\eps . $$
Since the map $(F_1 g, F_2 g) \mapsto (F_1 \cap F_2) g$ is well defined, 
we have that $C_q Z_n$ is determined by the pair $(C_p Z_n , C_\eps Z_n)$.
Also, since $C_q \leq C_p$, we have that $C_p Z_n$ is determined by $C_q Z_n$.
Thus, $\Pr$-a.s.,
$$ H(C_p Z_n) \leq H(C_q Z_n) \leq H(C_p Z_n) + H(C_\eps Z_n) . $$ 
Taking expectation (under $\E$) we obtain 
\eqref{eq:bound-for-the-entropy}, which then immediately leads to
$0 \leq h(q) - h(p) \leq h(\eps)$.

Corollary \ref{cor:entropy_semi_cont} asserts that the entropy function is upper-semi continuous. Thus, $\limsup_{\eps \to 0} h(\eps) \leq h(0)$.  Thus, if $h(0)=0$ we have that 
$\lim_{\eps \to 0} h(\eps) = 0$, and that the entropy function is continuous.
\end{proof}

\subsection{Applications to entropy realizations}



Let $\mu$ be a probability measure on a group $G$ and $N \triangleleft G$ a normal subgroup.
We use $\bar \mu$ to denote the projected measure on the quotient group $G / N$.

\begin{cor}\label{cor:single-subgroup-mu}
Let $\mu$ be a generating finite entropy probability measure on 
a discrete group $G$. 
Assume that there exists some $\s \leq G$ such that 
\begin{itemize}
\item $\rweq{\nO}=0$
\item $\rweq{\nI}=\rwe$
\end{itemize}
Then $(G,\mu)$ has a full realization. 

More generally, if only the first condition is satisfied, 
we get a realization of the interval $[0,\rweq{\nI}]$.
\end{cor}

\begin{proof}
Consider the path of IRSs $\lambda_{p,K}$ defined in Claim \ref{claim-definition-of-path}.
Let $h(p)$ be defined as above, before Proposition \ref{prop:continuity}.
Since $h(0) = \rweq{\nO}=0$, by
Proposition \ref{prop:continuity} we have that $h$ is continuous, and thus 
we have a realization of the interval $[h(0) , h(1) ] = [0, \rweq{\nI}]$.
\end{proof}

In the above corollary, the subgroup $\s$ may depend on the measure $\mu$.
However, it is possible in some cases to find a subgroup $\s$ that provides 
realization for {\em any} measure $\mu$, and we now discuss alternative
conditions for the corollary.

The conditions in Corollary~\ref{cor:single-subgroup-mu} implies that $\nO$ is
a ``large'' subgroup and $\nI$ is ``small'' - at least in the sense of the random walk entropy. The best scenario hence is when $\nO=G$ and $\nI=\{1\}$. 
The prototypical example is the following:

\begin{example}\label{example:sinfity}
Let $\sinf$ denote the group of all finitely supported permutations of some infinite
countable set $X$. 
Fix some $x_0\in X$ and define $K$ to be the subgroup of 
all the finitely supported permutations of $X$ that stabilize $x_0$;
{\em i.e.}\ $K=\mathrm{stab}(x_0) = \{ \pi \in \sinf \ : \ \pi(x_0)=x_0 \}$. 
Conjugations of $K$ are of the form $\mathrm{stab}(x)$ for $x\in X$.

It follows that $\nI=\{e\}$, because $\sinf$ acts transitively on $X$ 
and only the trivial permutation stabilizes all points $x \in X$.

Since each element of $\sinf$ is a finitely supported permutation, 
it stabilizes all but finitely many points in $X$.  Thus, $\nO=\sinf$. 

This is actually a proof of Theorem~\ref{thm:main-sinf}:
$(\sinf , \mu)$ admits full realization for any generating, finite entropy, probability measure $\mu$.
\end{example}

Note that $\sinf$ is a countable group which is not finitely generated.
However, by adding some elements to it, we can get a finitely generated group with
a similar result, as we explain below in Section \ref{sec:lamplighter}. 

To continue our discussion, we want to weaken the condition $\nO=G$ to some
algebraic condition that will
still guarantee that $\rweq{\nO}=0$ for any $\mu$. 

Given a subgroup $\s$ (or a conjugacy class $\s^G$) we say that $\s$ is \define{locally co-nilpotent}
in $G$ if $G / \nO$ is nilpotent.  (The reason for this name will become more apparent in Section
\ref{scn:local}.)  With this definition, Proposition~\ref{thm:no-gap} follows immediately, as may be seen by the following.

\begin{cor}\label{cor:single-subgroup}
Assume that $G$ admits a subgroup $\s \leq G$ such that
$G/\nO$ is a nilpotent group.
Then, for any generating finite entropy probability measure $\mu$,
the pair $(G,\mu)$ admits a realization of the interval $[0, \rweq{\nI}]$.

If, in addition $\rweq{\nI}=\rwe$ (\emph{e.g.} if $\nI = \{e\}$) then $(G,\mu)$ admits full realization.
\end{cor}

\begin{proof}
When $G/\nO$ is nilpotent, the Choquet-Deny Theorem (Theorem \ref{fact:nilpotent-are-Liouville})
tells us that the 
first condition of Corollary \ref{cor:single-subgroup-mu} is satisfied. 
\end{proof}

In Section \ref{sec:lamplighter} we apply this to obtain
full realizations for a class of lamplighter groups as well as to some 
extensions of the group of finitely supported permutations.
However, one should not expect to find a subgroup 
satisfying the conditions of Corollary \ref{cor:single-subgroup} in every group.
These conditions are quite restrictive, for example we the following:

\begin{lemma}\label{lem:amenability-restriction}
If $G$ is a group with a self-normalizing subgroup $K$ that satisfies the two
conditions of Corollaries \ref{cor:single-subgroup-mu} (or \ref{cor:single-subgroup}) then $G$ is amenable.
\end{lemma}
\begin{proof}
First, observe that $\rweq{\nO}=0$ implies that $\nO$ is a co-amenable (normal) subgroup in $G$. Indeed, it is well know that a group, $G/\nO$ in this case, admits a random walk
which is Liouville is amenable.

Next, the condition $\rweq{\nI}=\rwe$ implies that $\nI$ is an amenable group.
This follows from Kaimanovich's amenable extension theorem~\cite{kaimanovich2002poisson}: Let $\mu$ be a random walk on $G$. The Furstenberg-Poisson boundary of $G/\nI$ with the $\mu$-projected random walk is a $(G,\mu)$-boundary, or, in other words a $G$-factor of the 
Furstenberg-Poisson boundary $\Pi(G,\mu)$. Since it has the same entropy as
the random walk entropy, this factor is actually isomorphic to $\Pi(G,\mu)$.
Now Kaimanovich's theorem asserts that $\nI$ is an amenable group.

To conclude that $G$ is amenable we show that $\nO / \nI$ is amenable, showing
that the normal subgroup $\nI$ is both amenable and co-amenable in $G$.

In general, we claim that $\nO / \nI[N_G(K)]$ is always an amenable group. This
will complete the proof, because we assumed that $N_G(K) = K$.

Indeed, $G$ acts by permutations on the conjugacy class $K^G$.
Note that any element in $\nO$ has finite support as a permutation on $K^G$.
Also, 
it may be checked that the elements of $\nO$ that stabilize every $K^\gamma$ 
are precisely the elements of $\core_G(N_G(K))$.
Thus, $\nO / \nI[N_G(K)]$ is isomorphic to a subgroup 
of the amenable group $\sinf$.
\end{proof}


The subgroups of the free group that we will consider in 
Section \ref{sec:free-group} (to prove Theorem \ref{thm:main-free})
are all self-normalizing.  Since the free group is non-amenable, 
we cannot obtain the conditions of Corollary \ref{cor:single-subgroup-mu}
simultaneously in this fashion. 
For this reason, we approximate the second condition using a sequence of subgroups.

\begin{cor}\label{cor:sequence-of-subgroups}
Let $\mu$ be a generating finite entropy probability measure on a discrete group $G$. 
Assume that there exists a sequence of subgroups $\s_n \leq G$ such that:
\begin{enumerate}
\item \label{condition1} $\rweq{\nO[\s_n]} = 0$
({\em e.g.}\ whenever $G/ \nO[\s_n]$ are nilpotent),
\item \label{condition2} $\rweq{\nI[\s_n]} \to \rwe$ 
\end{enumerate}
Then $(G,\mu)$ admits a full realization.
\end{cor}

\begin{proof}
By the first condition, for any fixed $n$, the entropy along the path $\lambda_{p,K_n}$ gives
a realization of the interval $[0,\rweq{\nI[\s_n]}]$.
By the second condition, we conclude that any number in $[0,\rwe]$ is realized as the 
entropy some IRS of the form $\lambda_{p,K_n}$.
\end{proof}

\subsection{Examples: lamplighter groups and permutation groups} 
\label{sec:lamplighter}

We are now ready to apply our tools to get Theorem \ref{thm:main-lamplighter}, that
is, full realization for lamplighter groups with Liouville base.

\begin{proof}[Proof of Theorem \ref{thm:main-lamplighter}]
Choose the subgroup $\s = \oplus_{B\backslash \{e\}} L \rtimes \{e\}$ where $e$ is 
the trivial element in the base group $B$. It may be simply verified that 
$\nI = \{e\}$ and $\oplus_{B} L \rtimes \{e\} \lhd \nO$.

Whenever $(B, \bar \mu)$ is Liouville, 
also $(G / \nO , \bar \mu)$ is.
This is because $B \cong G / (\oplus_{B} L \rtimes \{e\})$, so $G/\nO$
is a quotient group of $B$.
Thus, 
we have that $\rweq{\nO} = 0$, 
and we obtain a realization of the full interval $$[0, \rweq{\nI}] = [0, \rwe]. $$
\end{proof}

We now turn to discuss extensions of the group of finitely supported permutations.
Let $X$ be a countable set, and denote by $\fullsinf$ the group of all
permutations of $X$. Recall that $\sinf$ is its subgroup of all \emph{finitely
supported} permutations.
While $\sinf$ is not finitely generated, one can add elements from $\fullsinf$ to get a finitely generated group. A standard example is when adding a ``shift'' element $\sigma\in\fullsinf$, that is, $\sigma$ acts as a free transitive
permutation of $X$. Let $\Sigma=\left\langle\sigma \right\rangle$ be the subgroup 
generated by $\sigma$.
Formally the element $\sigma$ acts on $\sinf$ by conjugation (recall that both groups are subgroups of $\fullsinf$) and we get a finitely generated group $\sinf \rtimes \Sigma$.

One can replace $\Sigma$ by any other subgroup of $\fullsinf$. The reason that we
want to change $\Sigma$ by a ``larger'' group is to have a non-Liouville group $\sinf \rtimes \Sigma$.

\begin{thm}
Let $\Sigma \le \fullsinf$ and consider $G=\sinf \rtimes \Sigma$. Let $\mu$ be a generating finite entropy 
probability measure on $G$, such that the the projected measure $\bar{\mu}$ on 
$\Sigma \cong G / \sinf \rtimes \{e\}$ is Liouville. 
Then $(G,\mu)$ has full realization.

In particular, if $\Sigma$ is a nilpotent group, $(G,\mu)$ admits full realization for any $\mu$.
\end{thm}

\begin{proof}
Fix some point $x_0\in X$ and let $K= \mathrm{stab}(x_0) \rtimes \{e\}$ where 
$\mathrm{stab}(x_0)\le \sinf$ is the stabilizer of $x_0$. It is easy to verify that $\nI = \{e\}$
and that $\nO = \sinf \rtimes \{ e\}$, hence by the assumption $\rweq{\nO}=0$ and
by Corollary \ref{cor:single-subgroup-mu} we get a full realization.
\end{proof}

\begin{remark}
Depending of $\Sigma$ and $\mu$, it might be that $(G,\mu)$ is Liouville where $G=\sinf \rtimes \Sigma$. In that case, this result of full realization is trivial
as the only realizable value is $0$. 

However, it is easy to find nilpotent $\Sigma$ such 
that $G$ is finitely generated and that $(G,\mu)$ is not Liouville (say, for any generating, symmetric, finitely supported $\mu$).
\end{remark}

%

\section{Full realization for the free group}\label{sec:free-group}

We now turn to discuss our realization results for $G=\freegroup[r]$ the free group of rank $r$ for any $r\ge 2$. 
It suffices to prove that the free group satisfies the conditions of Corollary~\ref{cor:sequence-of-subgroups}. We construct the subgroups $\s_n$ by describing their Schreier graphs.
For simplicity of the presentation 
we describe the result only for the free group on two generators, 
$\freegroup = \langle a,b \rangle$. 
The generalization to higher rank is straightforward.

\subsection{Schreier graphs} 
\label{scn:Schreier}

Recall our notations for Schreier graphs from Section~\ref{sec:schreier-graphs}.
It will be convenient to use $|v|$ to denote the graph distance of a vertex $v$ to the root
 vertex in a fixed Schreier graph.

\subsection{Fixing}

In this section we describe a condition on Schreier graphs that we call {\em fixing}.
For the normal core of the associated subgroup, this condition will ensure that 
the random walk entropy of the quotient approximates the full random walk entropy as
required by Condition 2 in Corollary \ref{cor:sequence-of-subgroups} (see Proposition 
\ref{prop:growing fixing general} below).

Let $\Lambda$ be a Schreier graph with root $o$.  

For a vertex $v$ in $\Lambda$, the \define{shadow} of $v$, denoted $\shd(v)$ is
the set of vertices $u$ in $\Lambda$ such that any path from $u$ to the root $o$ must pass through $v$.

For an integer $n$ define the finite subtree
$(\freegroup)_{\leq n}$ to be the induced (labeled) subgraph on
$\set{ g \in \freegroup \ : \ |g| \leq n }$. 
We say that 
$\Lambda$ is \define{$n$-tree-like} if the ball of radius $n$ about $o$ in $\Lambda$ is isomorphic 
to $(\freegroup)_{\leq n}$, and if every vertex at distance $n$ from $o$ has a non-empty shadow.
Informally, an $n$-tree-like graph is one that is made up by gluing 
disjoint graphs to the leaves of a depth-$n$ regular tree.

\begin{lemma} \label{lem:covering}
Let $C \leq K \leq \freegroup$ be subgroups.
Let $\Lambda_C , \Lambda_K$ be the Schreier graphs of $C,K$ respectively.

If $\Lambda_K$ is $n$-tree-like, then $\Lambda_C$ is also $n$-tree-like.
\end{lemma}

\begin{proof}
This follows from the fact that 
$C \leq K$ implies that $\Lambda_C$ is a cover of $\Lambda_K$ 
(see {\em e.g.}\ \cite[Lemma 20]{leemann2016}). 
\end{proof}

For an $n$-tree-like $\Lambda$, since any vertex is connected to the root, every vertex at
distance greater than $n$ from the root $o$ is in the shadow of some vertex at distance
precisely $n$.
For a vertex $v$ in $\Lambda$ with $|v|>n$, and $k\le n$ we define the \define{$k$-prefix} of $v$ in $\Lambda$ to be $\pref_k(v) = \pref_{k,\Lambda}(v) = u$ for the unique $u$ such that both $|u| = k$ and $v \in \shd(u)$.
If $|v| \leq k$ then for concreteness we define $\pref_k(v) = v$. 

Since $n$-tree-like graphs look like $\freegroup$ up to distance $n$, the behavior of 
$\pref_k$ on $n$-tree-like graphs where $n \geq k$ is the same as on $\freegroup$.
This is the content of the next elementary lemma, whose proof is left as an exercise.

\begin{lemma}
\label{lem:prefix tree like}
Let $\Lambda$ be a $n$-tree-like Schreier graph.
Let $\varphi: \Lambda \to \freegroup$ be a function mapping the ball of radius $n$ in $\Lambda$
isomorphically to $(\freegroup)_{\leq n}$.

Then, for any vertex $|v| \leq n$ in $\Lambda$ we have that 
$$ \pref_{k,\Lambda}(v) = \varphi^{-1} \pref_{k,\freegroup}(\varphi(v)) . $$
\end{lemma}


We use $(Z_t)_t$ to denote the $\mu$-random walk on $\freegroup$ 
and $(\tilde Z_t)_t$ the projected walk on a Schreier graph $\Lambda$ with root $o$.
(So $\tilde Z_t=K Z_t$ where $K$ is the associated subgroup to the root $o$ in the Schreier graph $\Lambda$.) 

Another notion we require is the \define{$k$-prefix at infinity}. 
Define the random variable $\pref_k (\tilde Z_\infty) = \pref_{k,\Lambda} (\tilde Z_\infty)$, 
taking values in the set of vertices of distance $k$ from $o$ in $\Lambda$, as follows:
If the sequence $\pref_k (\tilde Z_t )_t$ stabilizes at $u$, that is, if there exists $t_0$ such that 
for all $t>t_0$ we have $\pref_k( \tilde Z_t) = \pref_k ( \tilde Z_{t-1}) = u$, 
then define $\pref_k(\tilde Z_\infty) : = u$.
Otherwise, define $\pref_k(\tilde Z_\infty) = o$.
Note that for an $n$-tree-like graph with $n>0$ we have that $\pref_k(\tilde Z_\infty) = o$ (almost surely) if and only if the walk $(\tilde Z_t)_t$ returns to the ball of radius $k$ infinitely many times.


\begin{definition}
Let $0 < k < n$ be natural numbers and $\alpha \in (0,1)$.
We say that $\Lambda$ is \define{$(k,n,\alpha)$-fixing} if 
$\Lambda$ is $n$-tree-like and for any $|v| \geq n$,
$$ \Pr [ \forall \ t \ , \ \pref_k (\tilde Z_t) =\pref_k(\tilde Z_0) \ | \ \tilde Z_0 = v ] \geq \alpha . $$
\end{definition}

That is, with probability at least $\alpha$, for any $|v|\geq n$,
the projected random walk started at $v$ never leaves $\shd( \pref_k(v) )$.
Hence, the random walk started at depth $n$ in the graph fixes the $k$-prefix with probability at least $\alpha$.
Note that this definition depends on the random walk $\mu$.

A good example of a fixing graph is the Cayley graph of the free group itself.

\begin{lemma}
\label{lem:free group fixing}
Let $\mu$ be a finitely supported generating probability measure on $\freegroup$.
Then, there exists $\alpha = \alpha(\mu) >0$ such that the Cayley graph of $\freegroup$
is $(k,k+1,\alpha)$-fixing for any $k>0$.
\end{lemma}

\begin{proof}
It is well known that a $\mu$-random walk $(Z_t)_t$ on $\freegroup$ is transient
(for example, this follows from non-amenability of $\freegroup$).
Set $r = \max \{ |g| \ : \ \mu(g) > 0 \}$ be the maximal jump possible by the $\mu$-random walk.
In order to change the $k$-prefix, the walk must return to distance 
at most $k+r$ from the origin.  That is, if $|Z_t| > k+r$ for all $t >t_0$, we have that
$\pref_k(Z_t) = \pref_k(Z_{t_0})$ for all $t > t_0$.

By forcing finitely many initial steps (depending on $r,\mu$), 
there exists $t_0$ and $\alpha>0$ such that 
$\Pr [ A \ | \ |Z_0| = k+1 ] \geq \alpha$ where 
$$  A= \{\forall 0<t \leq t_0 \ , \ \pref_k(Z_t) = \pref_k(Z_0) \textrm{ and } |Z_{t_0}| >k+r \} . $$
By perhaps making $\alpha$ smaller, using transience, we also have that 
$$ \Pr [ \forall \ t > 0 \ , \ |Z_t| \geq |Z_0| ] \geq \alpha . $$
Combining these two estimates, with the Markov property at time $t_0$,
\begin{align*}
\Pr [ \forall t  , \ \pref_k(Z_t) = \pref_k(Z_0) \ | \ |Z_0|=k+1 ]  
& \geq \Pr [ A \ \mbox{ and } \{ \forall \ t > t_0 \ , \ |Z_t| \geq |Z_{t_0}| \} ] \\
& \geq \Pr [ A ] \cdot  \Pr [ \forall \ t > 0 \ , \ |Z_t| \geq |Z_0| ] \geq \alpha^2 .
\end{align*}
%
This is the definition of $\freegroup$ being $(k,k+1,\alpha)$-fixing.
\end{proof}

A very useful property of fixing graphs is that when $\alpha$ is close enough to $1$, 
there exists a finite random time - the first time the random walk leaves the tree area - on which we can already guess with a high accuracy the $k$-prefix at infinity of the random walk. A precise formulation is the following.

\begin{lemma}
\label{lem:close entropy}
Let $\Lambda$ be a $(k,n,\alpha)$-fixing Schreier graph.
Let $(Z_t)_t$ be a $\mu$-random walk on $\freegroup$ and $(\tilde Z_t)_t$ 
the projected walk on $\Lambda$.
Define the stopping time $$T = \inf \{ t \ : \ |\tilde Z_t| \geq n \} . $$

Then, 
$$ \big| H( Z_1 \ | \ \pref_k(\tilde Z_{T}) ) - H(Z_1 \ | \ \pref_k(\tilde Z_\infty) ) \big| \leq 
\eps(\alpha) + 2 (\log 4) (1-\alpha) k , $$
where $0 < \eps(\alpha) \to 0$ as $\alpha \to 1$.
\end{lemma}

\begin{proof}
First note that a general entropy inequality is
\begin{align*}
H(X \ | \ A) - H(X \ | \ B) 
\leq H(A \ | \ B) + H(B \ | \ A) ,
\end{align*}
which for us provides the inequality
\begin{align*}
\big| H( Z_1 \ | \ \pref_k(\tilde Z_{T}) ) & - H(Z_1 \ | \ \pref_k(\tilde Z_\infty) )    \big| 
\\
& \leq  H(\pref_k(\tilde Z_{T})  \ | \ \pref_k(\tilde Z_\infty) ) + 
H(\pref_k(\tilde Z_{\infty})  \ | \ \pref_k(\tilde Z_T) ) .
\end{align*}
We now bound this last quantity using Fano's inequality (see {\em e.g.}\ 
\cite[Section 2.10]{cover2012elements}).
Taking $\tilde \alpha = \Pr [ \pref_k(\tilde Z_{T}) = \pref_k(\tilde Z_\infty) ]$, 
since the support of both 
$\pref_k(\tilde Z_{\infty}) , \pref_k(\tilde Z_T)$ is of size $4\cdot 3^{k-1} < 4^k$,
\begin{align*}
H(\pref_k(\tilde Z_{\infty})   \ | \ \pref_k(\tilde Z_T) ) & \leq H(\tilde \alpha , 1- \tilde \alpha) +
(1-\tilde \alpha) \log 4^k ,
\\
H(\pref_k(\tilde Z_{T})   \ | \ \pref_k(\tilde Z_\infty) ) & \leq H(\tilde \alpha , 1- \tilde \alpha) +
(1-\tilde \alpha) \log 4^k .
\end{align*}
where here we use the usual notation $H(\tilde \alpha , 1- \tilde \alpha)$ to denote the Shannon entropy of the probability vector $(\tilde \alpha , 1- \tilde \alpha)$.

Since $p \mapsto H(p,  1-p)$ decreases for $p \geq \tfrac12$, it suffices to prove that 
$\tilde \alpha = \Pr [ \pref_k(\tilde Z_{T}) = \pref_k(\tilde Z_\infty) ] \geq \alpha$.

But now, the very definition of $(k,n,\alpha)$-fixing implies that for any relevant $u$,
\begin{align*}
\Pr [ \pref_k & ( \tilde Z_\infty) = u \ | \ \pref_k (Z_T) = u ]  \\
& \geq \Pr [ \forall \ t > 0 \  , \ \pref_k(\tilde Z_{T+t}) \in \shd( \pref_k(u) ) \ | \ \pref_k(Z_T) = u ] \geq \alpha , 
\end{align*}
where we have used the strong Markov property at the stopping time $T$. 
Averaging over the relevant $u$ implies that $\tilde \alpha \geq \alpha$.
\end{proof}

The following is the main technical estimate of this subsection.
\begin{lemma}
\label{lem:growing fixing}
Fix a generating probability measure $\mu$ on $\freegroup$ with finite support.
Let $(\Gamma_j)_j$ be a sequence of Schreier graphs 
such that $\Gamma_j$ is $(k_j, n_j,\alpha_j)$-fixing.
Let $K_j$ be the subgroup corresponding to $\Gamma_j$.

Suppose that $k_j \to \infty$, $n_j-k_j \to \infty$ and $(1-\alpha_j) k_j \to 0$ as $j \to \infty$.
Then, 
$$ \limsup_{j \to \infty} h_{RW} (\freegroup / \core_{\freegroup} (K_j) , \bar \mu ) = h_{RW}(\freegroup , \mu) . $$
\end{lemma}

\begin{proof}
Let $(Z_t)_t$ be a $\mu$-random walk on $\freegroup$.

Take $j$ large enough so that $n_j-k_j$ is much larger that $r=\max \{ |g| \ : \ \mu(g)>0 \}$.
To simplify the presentation we use
$\Gamma_j = \Gamma, K = K_j, C =C_j =  \core_{\freegroup} (K_j)$ 
and $k=k_j,n= n_j, \alpha= \alpha_j$, omitting the subscript $j$ when it is clear from the context.
Let $\Lambda = \Lambda_j$ be the Schreier graph corresponding to $C = C_j$.

Let $\Tail_C = \bigcap_t \sigma(C Z_t , C Z_{t+1} , \ldots )$ 
and $\Tail = \bigcap_t \sigma(Z_t , Z_{t+1} , \ldots )$
denote that tail $\sigma$-algebras of the random walk on $\Lambda$ and on $\freegroup$ respectively.
Kaimanovich and Vershik (eq.\ 13 on p.\ 465 in \cite{kaimanovich1983random}) show that 
$$ h_{RW} (\freegroup / C , \bar \mu ) = H(C Z_1) - H(C Z_1 \ | \ \Tail_C ) $$
and 
$$ \rwe = H(Z_1) - H(Z_1 \ | \ \Tail ) . $$
Now, since $C \leq K$, and since $\Gamma$ is $n$-tree-like, also $\Lambda$ is (Lemma \ref{lem:covering}).
Since $\Lambda$ is $n$-tree-like, 
$C Z_1$ and $Z_1$ determine one another (Lemma \ref{lem:prefix tree like}), 
implying $H(C Z_1) = H(Z_1)$ 
and $H( CZ_1 \ | \ \Tail_C) = H(Z_1 \ | \ \Tail_C)$.
Since $\Tail_C \subset \Tail$, the inequality $H(Z_1 \ | \ \Tail_C) \geq H(Z_1 \ | \ \Tail )$ is immediate.
Hence, we only need to show that $\liminf_{j} H(Z_1 \ | \ \Tail_{C_j} ) \le H(Z_1 \ | \ \Tail)$.
 
Set $T = \inf \{ t \ : \ |C Z_t | \geq n \}$ as in Lemma \ref{lem:close entropy}.
Since $\Lambda$ is $n$-tree-like, 
$\pref_k(C Z_{T-1})$ and $\pref_k(Z_{T-1})$ determine one another (Lemma \ref{lem:prefix tree like}).
Also, since $n-k$ is much larger than $r$,
and specifically much larger than any jump the walk can make in one step, we know that
$\pref_k( C Z_T) = \pref_k (C Z_{T-1})$ and $\pref_k(Z_T) = \pref_k(Z_{T-1})$.
All in all, $H(Z_1 \ | \ \pref_k( C Z_{T} ) ) = H(Z_1 \ | \ \pref_k( Z_{T}) )$.

Now, set $\eps = \eps(\alpha,k) = 2 H(\alpha,1-\alpha) + 2 (\log 4) (1-\alpha) k$.
Since $\pref_k( CZ_\infty )$ is measurable with respect to $\Tail_C$,
by using Lemma \ref{lem:close entropy} twice, once for $\Lambda$ and once
for the tree, we get the bound 
\begin{align*}
H(Z_1 \ | \ \Tail_C ) & \leq H(Z_1 \ | \ \pref_k(C Z_\infty) ) \leq H(Z_1 \ | \ \pref_k( C Z_{T} ) ) + \eps
\\
& = H(Z_1 \ | \ \pref_k(Z_{T} ) ) + \eps \leq H(Z_1 \ | \ \pref_k(Z_\infty) ) + 2 \eps .
\end{align*}

Now, under the assumptions of the lemma, $\eps \to 0$ as $j \to \infty$.
Also, since $k_j \to \infty$ we have that 
$$ H(Z_1 \ | \ \pref_{k_j}(Z_\infty) ) \to H(Z_1 \ | \ \Tail ) . $$
Thus, taking a limit on the above provides us with
$$ \limsup_{j \to \infty} H(Z_1 \ | \ \Tail_{C_j} ) \leq 
\limsup_{j \to \infty} H(Z_1 \ | \ \pref_{k_j}(Z_\infty) ) = H(Z_1 \ | \ \Tail) . $$

Hence, as mentioned, we conclude that $\limsup_{j} H(Z_1 \ | \ \Tail_{C_j} ) = H(Z_1 \ | \ \Tail)$, and since $H(Z_1)=H(CZ_1)$ we get that 
$$ \limsup_{j\to\infty} h_{RW} (\freegroup / \core_{\freegroup} (K_j) , \bar \mu ) = h_{RW}(\freegroup , \mu) . $$

\end{proof}

To complete the relevance of fixing to convergence of the random walk entropies, 
we conclude with the following.

\begin{prop}
\label{prop:growing fixing general}
Fix a generating probability measure $\mu$ on $\freegroup$ with finite support.
Let $(\Gamma_j)_j$ be a sequence of Schreier graphs 
such that $\Gamma_j$ is $(k_j, n_j ,\alpha_j)$-fixing.
Let $K_j$ be the subgroup corresponding to $\Gamma_j$.

Suppose that $k_j \to \infty$ and $\alpha_j \to 1$ as $j \to \infty$.
Then, 
$$ \limsup_{j \to \infty} h_{RW} (\freegroup / \core_{\freegroup} (K_j) , \bar \mu ) = h_{RW}(\freegroup , \mu) . $$ 
\end{prop}

\begin{proof}
Since we are only interested in $\limsup$ we may pass to a subsequence and show that the parameters along this subsequence satisfy  
the assumptions of Lemma \ref{lem:growing fixing}. Notice that the assumption $k_j \to\infty$ implies in particular that $n_j\to\infty$.

Note that by definition, if $\Lambda$ is $(k,n,\alpha)$-fixing, then it is also 
$(k',n',\alpha')$-fixing for any $k' \leq k, n' \geq n, \alpha' \leq \alpha$.

For any $m$ choose $j_m$ large enough so that for all $i \geq j_m$ we have 
both $\alpha_i > 1- m^{-2}$ and $k_i > 2m$.
The subsequence of graphs $(\Gamma_{j_m} )_m$ satisfies that
$\Gamma_{j_m}$ is $(m,n_{j_m} ,1-\tfrac{1}{m^{2}})$-fixing.
Since $n_{j_m} - m \geq k_{j_m} - m > m$,
Lemma \ref{lem:growing fixing} is applicable.
\end{proof}

\subsection{Gluing graphs}
\label{sec:gluing-graphs}

Let $\Lambda$ be a Schreier graph of the free group $\freegroup=\left\langle a,b \right\rangle$ and let 
$S=\{a,a^{-1},b,b^{-1}\}$. We say that the pair 
$(\Lambda,\edge)$ is \define{$s$-marked} 
if $\edge$ is a (oriented) edge in $\Lambda$ labeled by $s\in S$.

For $s \in S$ we define $\mathcal{N}_s$ to be the following graph:  the vertices are non-negative the integers $\N$.
The edges are given by $(x+1,x)$, each labeled by $s$, 
and self-loops labeled by $\xi\in\{a,b\} \setminus \{s,s^{-1} \}$ at each vertex.
In order to be a Schreier graph,
this graph is missing one outgoing edge from $0$ labeled by $s$.

\begin{figure}[h]
\begin{center}
\includegraphics[width=0.5\textwidth]{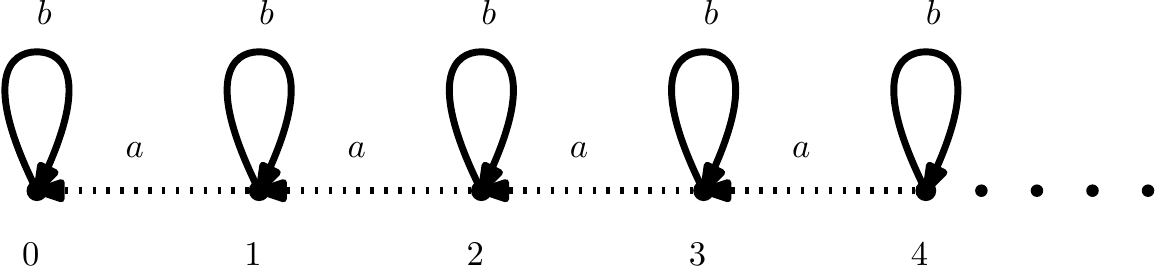}
\caption{
Part of the labeled graph $\mathcal{N}_a$.}
\end{center}
\end{figure}

Given a $s$-marked pair $(\Lambda,\edge)$ and an integer $n > 0$, we construct a Schreier graph 
$\Gamma_n(\Lambda,\edge)$:

The Cayley graph of $\freegroup$ is the $4$-regular tree. 
For an integer $n$ and a label (generator) $s$ recall the finite subtree
$(\freegroup)_{\leq n}$ which is the induced subgraph on
$\set{ g \in \freegroup \ : \ |g| \leq n }$, and let $(\freegroup)_n$ be the set of vertices $\set{ g \ : \ |g| =n}$.
For $g \in (\freegroup)_n$ there is exactly one outgoing edge incident to $g$ in $(\freegroup)_{\leq n}$.
If this edge is labeled by $\xi \in S$ we say that $g$ is a \define{$\xi$-leaf}.

\begin{figure}[h]
\begin{center}
\includegraphics[width=0.5\textwidth]{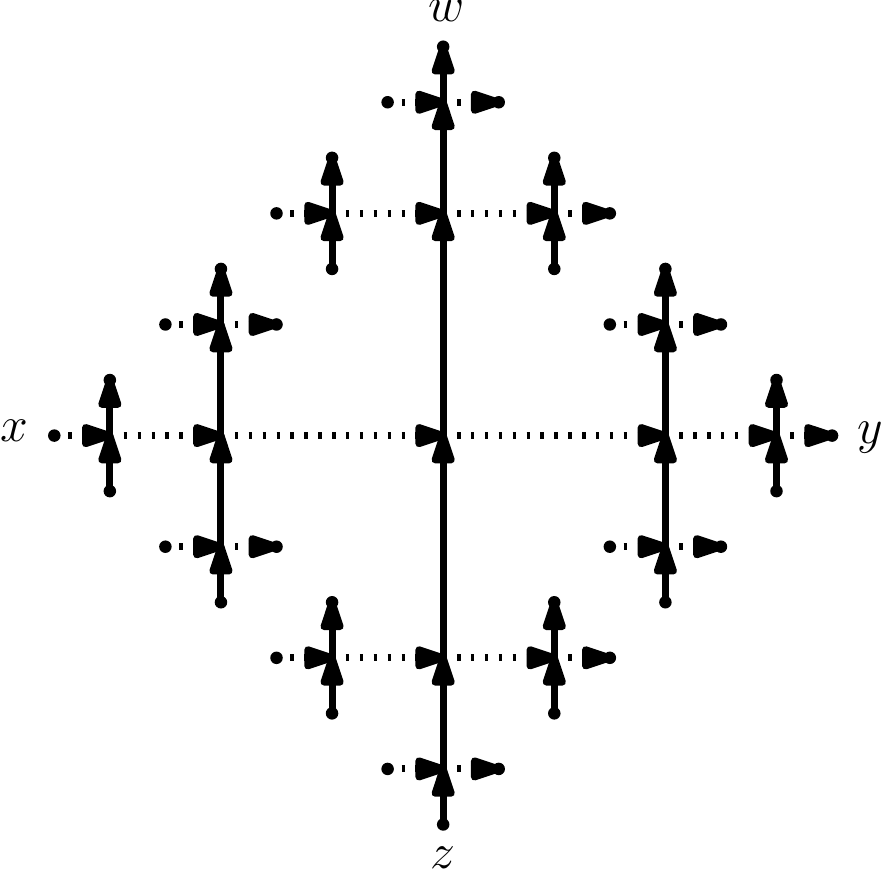}
\caption{
The subgraph $(\freegroup)_{\leq 3}$.  
Edges labeled $a$ are dotted and labeled $b$ are solid.
The vertices $x,y,z,w$ are $a,a^{-1},b,b^{-1}$-leaves respectively.
}
\end{center}
\end{figure}

For a $\xi$-leaf $g$, in order to complete $(\freegroup)_{\leq n}$ into a Schreier
graph, we need to specify $3$ more outgoing edges, with the labels $\xi^{-1}$ and 
the two labels from $S \setminus \{ \xi , \xi^{-1} \}$.

Now, recall our $s$-marked pair $(\Lambda,\edge)$.  
Let $g$ be a $\xi$-leaf for $\xi \not\in \{s, s^{-1} \}$.
Let $(\Lambda_g,\edge_g)$ be a copy of $(\Lambda,\edge)$, and suppose that $\edge_g = (x_g,y_g)$.
Connect $\Lambda_g$ to $g$ by deleting the edge $\edge_g$ from $\Lambda_g$
and adding the directed edges $(x_g,g) , (g,y_g)$ labeling them both $s$.
Also, let $\mathcal{N}_g$ be a copy of $\mathcal{N}_{\xi}$ above, 
and connect this copy to $g$ by a directed edge $(0_g,g)$ labeled by $\xi$
(here $0_g$ is the copy of $0$ in $\mathcal{N}_g$).
This takes care of $\xi$-leaves for $\xi \not\in \{s,s^{-1} \}$.

If $\xi \in \{ s,s^{-1} \}$, the construction is simpler:
For any $\xi$-leaf $\gamma$ (where $\xi \in \{ s,s^{-1} \}$),
let $\mathcal{N}_{\gamma}$ be a copy of $\mathcal{N}_\xi$ (with $0_{\gamma}$ the copy of $0$ as before), 
and connect $\mathcal{N}_{\gamma}$ by an edge $(0_{\gamma},\gamma)$ labeled by $\xi$. 
Finally add an additional self-loop with the missing labels at $\gamma$ ({\em i.e.}\ the labels in 
$S \setminus \{s,s^{-1} \}$ each in one direction along the loop).

By adding all the above copies to all the leaves in $(\freegroup)_{\leq n}$, we obtain a Schreier graph, 
which we denote $\Gamma_n(\Lambda,\edge)$.
See Figure \ref{fig:glue} for a visualization 
(which in this case is probably more revealing than the written description).
It is immediate from the construction that $\Gamma_n(\Lambda,\edge)$ is $n$-tree-like.

\begin{figure}[h]

\begin{center}
\includegraphics[width=\textwidth]{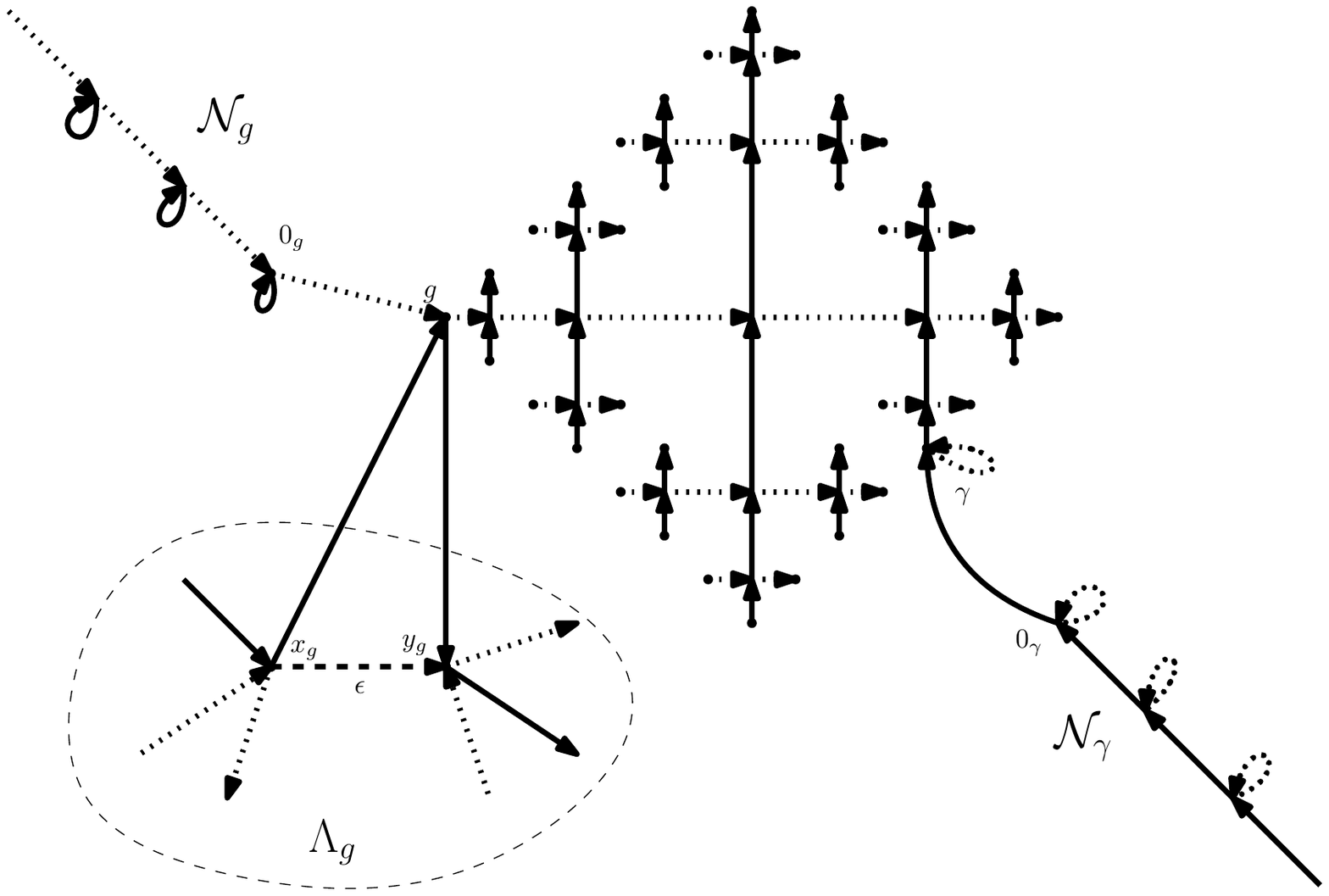}
\caption{A visiulization of gluing $(\Lambda,\edge)$ to $(\freegroup)_{\leq 3}$.
The dotted edges correspond to the label $a$, and the full edges to $b$.
The vertex $g$ is a $a$-leaf.  Thus, a copy of $\Lambda$, denoted $\Lambda_g$ is connected
by removing the edge $\edge$ (dashed) and adding to edges $(x_g,g) , (g,y_g)$ labeled $b$ (solid).
A copy of $\mathcal{N}_{a}$, denoted $\mathcal{N}_g$, is connected via an edge 
$(0_g,g)$ labeled $a$ (dotted). 
Also, to the $b$-leaf $\gamma$, a copy of $\mathcal{N}_b$, denoted $\mathcal{N}_\gamma$ is added
via an edge $(0_\gamma,\gamma)$ labeled $b$ (full), and an additional self-loop labeled $a$ (dotted) is added at 
$\gamma$.
\label{fig:glue}
}
\end{center}

\end{figure}

\subsection{Marked pairs and transience}
\label{scn:transience}

Recall the definition of a {\em transient} random walk (see {\em e.g.}\ \cite{Gabor, LyonsPeres}):
We say that a graph $\Lambda$ is \define{$\mu$-transient}
if for any vertex $v$ in $\Lambda$ there is some positive probability 
that the $\mu$-random walk started at $v$ will never return to $v$.

We now connect the transience of $\Lambda$ above 
to that of {\em fixing} from the previous subsection.

\begin{prop}  
\label{prop:trans and fix}
Let $\mu$ be a finitely supported generating probability measure on $\freegroup$.
Let $(\Lambda,\edge)$ be a $\mu$-transient $s$-marked pair.

Then, for any $\eps>0$ and $k$, there exist $n > k$ such that 
the graph $\Gamma_{n}(\Lambda,\edge)$ is $(k,n,1-\eps)$-fixing.
\end{prop}

The proof of this proposition is quite technical.
For the benefit of the reader, before proving the proposition we
first provide a rough sketch, outlining the main idea of the proof.

\subsubsection*{Proof Sketch}
Consider the graph $\Gamma_n(\Lambda,\edge)$.  This is essentially 
a finite tree with copies of $\Lambda$ and copies of $\mathbb{N}$ glued to the (appropriate) leaves.
At any vertex $v$ with $|v|>k$, there is a fixed positive probability $\alpha$ 
of escaping to the leaves without changing the $k$-prefix 
(because $\freegroup$ is $(k,k+1,\alpha)$-fixing).

Once near the leaves, with some fixed probability $\delta$ 
the walk reaches a copy of $\Lambda$ without
going back and changing the $k$-prefix. 
Note that although there are copies of the recurrent $\mathcal{N}_{\gamma}$ (e.g. for symmetric random walks) glued to some leaves, nonetheless, since $k$ is much smaller than $n$, there is at
least one copy of the transient $\Lambda$ in the shadow of $v$.

Finally, once in the copy of $\Lambda$, because of transience, 
the walk has some positive probability $\beta$ to escape to infinity without ever leaving $\Lambda$,
and thus without ever changing the $k$-prefix.
 
All together, with probability at least $\alpha \delta \beta$ the walk will never change the $k$-prefix.

If this event fails, the walk may jump back some distance, but not too much since we assume that $\mu$ 
has finite support.  So there is some $r$ for which the walk (deterministically) 
cannot retreat more than distance $r$,
even conditioned on the above event failing.

Thus, starting at $|v| > k + \ell r$ for large enough $\ell$, there are at least $\ell$ attempts to escape to infinity
without changing the $k$-prefix, each with a conditional probability of success at least $\alpha \delta \beta$.
By taking $\ell$ large enough, we can make the probability of changing the $k$-prefix as small as desired.

We now proceed with the actual proof.
\begin{proof}[Proof of Proposition \ref{prop:trans and fix}]
{\bf Step I.}
Let $(\tilde Z_t)_t$ denote the projected $\mu$-random walk on the Schreier graph $\Lambda$.
Assume that $\edge = (x,y)$, and note that since $\Lambda$ is $\mu$-transient, 
then 
$$ \Pr [ \forall \ t \ , \ \tilde Z_t \neq x \ | \ \tilde Z_0 = y ] + 
 \Pr [ \forall \ t \ , \ \tilde Z_t \neq y \ | \ \tilde Z_0 = x ] > 0 . $$
By possibly changing $\edge$ to $(y,x)$, without loss of generality we may assume that 
$$ \beta:=\Pr [ \forall \ t \ , \ \tilde Z_t \neq x \ | \ \tilde Z_0 = y ] > 0  . $$

Let $\ell >0$ be large enough, to be chosen below.
Set $n = k+ (\ell+2) r + 1$. 
We now change notation and use 
$(\tilde Z_t)_t$ to denote the projected walk on the Schreier graph $\Gamma_{n}(\Lambda,\edge)$. 
Denote $r = \max \{ |g| \ : \  \mu(g) > 0 \}$, the maximal possible jump of the $\mu$-random walk.
A key fact we will use throughout is that:
in order to change the $k$-prefix, 
the walk must at some point reach a vertex $u$ with $|u| \leq k + r$.

Let $\tau = \inf \{ t \ : \ n-r \leq |\tilde Z_t| \leq n \}$.
If we start from $|\tilde Z_0 | \leq n$ this is a.s.\ finite, since the jumps cannot be larger than $r$.
Now, by Lemma \ref{lem:free group fixing},
$\freegroup$ is $(k,k+1,\eta)$-fixing for all $k$ and some $\eta = \eta(\mu)>0$.
Since $\Gamma_{n}(\Lambda,\edge)$ is $n$-tree-like,
if $v$ is a vertex in $\Gamma_{n}(\Lambda,\edge)$ with $k<  |v| < n$, then 
\begin{align}
\label{eqn:tree fixing}
\Pr [  \pref_k(\tilde Z_\tau) = \pref_k(v)  \ | \ \tilde Z_0 = v ] \geq \eta
\end{align}

For $s \in S$ let $t_s$ be the smallest number such that $\mu^{t_s}(s) > 0$. 
Let $t_S = \max \{ t_s \ : \ s \in S \}$.
Let $\delta = \min \{ \mu^{t_s}(s) \ : \ s \in S \}$.
Let $v,u$ be two adjacent vertices in $\Gamma_{n}(\Lambda,\edge)$,
and assume that the label of $(v,u)$ is $s \in S$.
Note that the definitions above ensure that
$$ \Pr [ \tilde Z_{t_s} = u \ | \  \tilde Z_0=v ] \geq \delta . $$ 
Thus, when $|v| > k+t_S r$, 
with probability at least $\delta$ we can move from $v$ to $u$ in at most $t_S$ steps without 
changing the $k$-prefix.
This holds specifically for $|v| \geq n-r > k+t_S r$ by our choice of $n$ (as long as $\ell \geq t_S$).

Consider the graph $\Gamma_{n}(\Lambda,\edge)$.  Recall that it contains many copies 
of $\Lambda$ (glued to the appropriate leaves).  Let 
$\Lambda_1,\ldots,\Lambda_m$ be the list of these copies, and denote by 
$\edge_j = (x_j,y_j)$ the corresponding copies of $\edge$ in each.
Define $Y = \{ y_1, \ldots, y_m \}$.

Now, define stopping times
$$ U_m = \inf \{ t \ : \  |\tilde Z_t| < m  \}  \qquad \textrm{ and } \qquad 
T =\inf \{ t \ : \ \tilde Z_t \in Y \}  $$
(where $\inf \emptyset = \infty$).
Any vertex $v$ in $\Gamma_n (\Lambda, \edge)$ of depth 
$|v| < n$ must have some copy of $(\Lambda,\edge)$ in its shadow $\shd(v)$.
So there exists some path of length at most $n-|v|+1$ 
from $v$ into some copy of $(\Lambda,\edge)$, ending in some vertex in $Y$.  
If $|v| \geq n-r$ then we can use the strong Markov property at the first time the walk gets
to some vertex $u$ with $|u| < n$.  Since this $u$ must have $|u| \geq n-r$,
we obtain that for any $|v| \geq n-r$,
\begin{align*}
\Pr [ T < U_{n- r} \ | \ \tilde Z_0 = v ] \geq \delta^{r+1} .
\end{align*}
By $\mu$-transience of $\Lambda$, we have that starting from any $y_j \in Y$,
with probability at least $\beta$ the walk never crosses $(y_j,x_j)$, 
and so never leaves $\Lambda_j$ 
(and thus always stays at distance at least $n$ from the root).
We conclude that for any $|v| \geq n-r$,
\begin{align*}
\Pr [  \pref_k(\tilde Z_\infty) & = \pref_k(\tilde Z_0)  \ | \ \tilde Z_0 = v ] 
\geq \Pr [  U_{n-r} = \infty  \ | \ \tilde Z_0 = v ] \\
& \geq 
\Pr [ T < U_{n-r} \ | \ \tilde Z_0 = v ] \cdot \inf_{y_j \in Y} 
\Pr [  \forall t \ , \ \tilde Z_t \in \Lambda_j  \ | \ \tilde Z_0 = y_j ] 
\\
& \geq \delta^{r+1} \beta .
\end{align*}
(where we have used that $k+r < n-r$).
Combining this with \eqref{eqn:tree fixing}, 
using the strong Markov property at time $\tau$,
we obtain that for any $|v| > k$,
\begin{align}
\label{eqn:pref infty}
\Pr [ \pref_k(\tilde Z_\infty) & = \pref_k(v) \ | \ \tilde Z_0 = v ]
\geq \eta \delta^{r+1} \beta . 
\end{align}

{\bf Step II.}
Now, define the events 
$$ A_j = \{ \pref_k(\tilde Z_\infty) \neq \pref_k( \tilde Z_{U_{n- j r} } ) \} . $$
When $\tilde Z_0 = v$ for $|v| \geq n = k + (\ell+2) r + 1$ 
the event 
$\{ \pref_k(\tilde Z_\infty) \neq \pref_k(\tilde Z_0) \}$ implies that
$$ U_{|v|-r} < U_{|v|-2r} < \cdots < U_{|v|- \ell r} < U_{k+ r}  < \infty . $$
But, at time $U_{|v|-jr}$ we have that 
$$ | \tilde Z_{U_{|v|-jr} } | \geq | \tilde Z_{U_{|v|-jr}-1 } | - r \geq |v|-(j+1)r > k + r . $$
Thus, we have by \eqref{eqn:pref infty}, using the strong Markov property 
at time $U_{|v|-jr}$,
$$ \Pr [ A_{j+1} \ | \ \tilde Z_0 , \ldots, \tilde Z_{U_{|v|-jr} } ] \leq 1 - \eta \delta^{r+1} \beta , $$
implying that 
$$ \Pr [ A_{j+1} \ | \ (A_1)^c \cap \cdots \cap (A_j)^c ] \leq 1- \eta \delta^{r+1} \beta . $$
Thus, as long as $|v| \geq n$ we have that
\begin{align*}
\Pr [ \pref_k(\tilde Z_\infty) & \neq \pref_k (v) \ | \ \tilde Z_0 = v ] \leq 
\Pr [ A_1 \cup \cdots \cup A_{\ell} \ | \ \tilde Z_0 = v ]  
\leq (1-\eta \delta^{r+1} \beta)^{\ell} .
\end{align*}
Choosing $\ell$ such that $(1-\eta \delta^r \beta)^\ell < \eps$, 
we obtain that for any $|v| \geq n$ 
$$ \Pr [ \pref_k(\tilde Z_\infty) \neq \pref_k(v) \ | \ \tilde Z_0 = v ] < \eps . $$
That is, the graph $\Gamma_{n}(\Lambda,\edge)$ is $(k,n , 1-\eps)$-fixing.
\end{proof}

\begin{remark} \label{rem:finite supp}
Proposition \ref{prop:trans and fix} is the only place we require the measure $\mu$
to have finite support.  Indeed, we believe that the proposition should hold with 
weaker assumptions on $\mu$, perhaps even only assuming that $\mu$ is generating and has
finite entropy.  
If this is the case, we could obtain full realization results for the free group for such measures.
\end{remark}

\subsection{Local properties}
\label{scn:local}

Recall from Section \ref{scn:empty-core}
that given a subgroup $K \leq \freegroup$ we define
$\norm{g}$ as the number of cosets of the normalizer $b \in \n \backslash G$ for which 
$g \not\in K^b$.
Recall also that $\core_\emptyset(K) = \{ g \ : \norm{g} < \infty \}$.

In this subsection we will provide conditions on $K$ under which $\core_\emptyset(K)$ contains 
a given subgroup of $\freegroup$.  This will be useful in determining 
properties of $\freegroup / \core_\emptyset(K)$ (specifically, 
whether $\freegroup / \core_\emptyset(K)$
is nilpotent).

First, some notation:
Let $g \in \freegroup$.  Let $g = w_1 \cdots w_{|g|}$  be a reduced word with 
$w_i \in S$ for all $i$.
Let $\Lambda$ be a Schreier graph with corresponding subgroup $K \leq \freegroup$.
Fix a vertex $v$ in $\Lambda$.
Let $v_0(g)=v$ and inductively define $v_{n+1}(g)$ to be the unique vertex of $\Lambda$
such that $(v_n(g) , v_{n+1}(g))$ is labeled by $w_{n+1}$.
In other words, if $v = K \gamma$ then $v_i(g) = K \gamma w_1 \cdots w_i$.

Recall that $g \in K$ if and only if for $v=o$ the root of $\Lambda$. Hence,
for any $g\in K$, we have $v_{|g|}(g) = v_0(g)$.
From this we can deduce that when $v = K \gamma$, then 
$$ v_{|g|}(g) = v_0(g) \iff \gamma g \gamma^{-1} \in K \iff g \in K^\gamma . $$

\begin{definition}
Let $\Gamma_1, \ldots, \Gamma_n$ 
be some Schreier graphs with roots $\rho_1 ,\ldots, \rho_n$.  
Let $B_\Lambda(v,r)$ be the ball of radius $r$
around $v$ in the graph $\Lambda$ (and similarly for $\Gamma_j$).
We say that $\Lambda$ (or sometimes $K$) is \define{locally-$(\Gamma_1,\ldots,\Gamma_n)$} if 
for any $r>0$ there exists $R>0$ such that for all vertices $v$ in $\Lambda$ with $|v|>R$
we have that the ball $B_\Lambda(v,r)$ 
is isomorphic (as a labeled graph) to at least 
one of the balls $B_{\Gamma_1}(\rho_1,r) , \ldots, B_{\Gamma_n}(\rho_n,r)$.
\end{definition}

That is, a subgroup $K$ is locally-$(\Gamma_1,\ldots, \Gamma_n)$ if, after
ignoring some finite area, locally we see one of the graphs $\Gamma_1, \ldots, \Gamma_n$.

The main purpose of this subsection is to prove:
\begin{prop}
\label{prop:local}
Let $K \leq \freegroup$.  
Let $G_1 , \ldots, G_n \leq \freegroup$ be subgroups
with corresponding Schreier graphs $\Gamma_1 , \ldots, \Gamma_n$.

If $K$ is locally-$(\Gamma_1, \ldots, \Gamma_n)$ then 
$G_1 \cap \cdots \cap G_n \leq \core_\emptyset(K)$.
\end{prop}

\begin{proof}
Let $g \in G_1 \cap \cdots \cap G_n$.
We need to prove that $|| g ||_K < \infty$.

Let $\Lambda$ be the Schreier graph of $K$.
Let $r = |g|$.  Let $R$ be large enough so that for $|v|>R$ in $\Lambda$
the ball $B_\Lambda(v,r)$ is isomorphic to one of $B_{\Gamma_j}(\rho_j,r)$
(where as before, $\rho_j$ is the root vertex of $\Gamma_j$).

Let $|v| > R$ and 
let $j$ be such that $B_{\Lambda}(v,r)$ is isomorphic to $B_{\Gamma_j}(\rho_j,r)$.
Set $\Gamma=\Gamma_j, \rho = \rho_j$.

Consider the paths $\rho = \rho_0(g) , \ldots, \rho_{r}(g)$ in $\Gamma$ and 
$v=v_0(g)  , \ldots, v_{r}(g)$ in $\Lambda$.
Since these each sit in $B_\Gamma(\rho,r)$ and $B_\Lambda(v,r)$ respectively, 
we have that $\rho_r(g) = \rho_0(g)$ if and only if $v_r(g) = v_0(g)$.
Thus, $g \in G_j$ implies that $g \in K^\gamma$ where $v = K \gamma$.

In conclusion, we have shown that if $\gamma$ is such that $| K \gamma| > R$ 
then $g \in K^\gamma$.  
Hence, there are only finitely many $\gamma \in \freegroup$ such that $g \not\in K^\gamma$,
which implies that $|| g ||_K < \infty$.
\end{proof}

\subsection{Full realization}
\label{scn:full real}

In this subsection we prove Theorem \ref{thm:main-free}.

In light of Corollary \ref{cor:sequence-of-subgroups} and
Proposition \ref{prop:growing fixing general},
in order to prove full realization for $\freegroup$,
we need to find a sequence of subgroups $K_n$ with Schreier graphs $\Gamma_n$ such that
the following properties hold:
\begin{itemize}
\item $\Gamma_n$ is $(k_n,k'_n,\alpha_n)$-fixing, with $k_n \to \infty$ and $\alpha_n \to 1$.
\item The subgroups $\nO[K_n]$ are co-nilpotent in $\freegroup$.
\end{itemize}

To show the first property, we will use the gluing construction from Proposition \ref{prop:trans and fix},
by finding a suitable marked pair $(\Lambda,\edge)$ that is $\mu$-transient.

\begin{lemma}
\label{lem:locally co-nilpotent}
Let $N$ be a normal co-nilpotent subgroup in $\freegroup$, and let $\Lambda$ be
its associated Schreier graph. Choose some edge $\edge$ in $\Lambda$ labeled with $s \in S$ and
consider $(\Lambda,\edge)$ as a marked pair.
For any $n>0$ let $K_n \leq \freegroup$ be the subgroup corresponding to the root 
of the Shcreier graph $\Gamma_n(\Lambda,\edge)$.

Then, for any $n>0$, the normal subgroup $\core_\emptyset (K_n)$ is co-nilpotent in $\freegroup$.
\end{lemma}

\begin{proof}
Let $G_0 = \freegroup,  G_{j+1} = [G_j,\freegroup]$ be the descending central series of $\freegroup$.
Since $N$ is co-nilpotent, there exists $m$ such that $G_m \lhd N$.

By definition of $\Gamma_n(\Lambda,\edge)$, outside a ball of radius $n$ in $\Gamma_n(\Lambda,\edge)$,
we only have glued copies of $\Lambda$ or of $\mathcal{N}_s$ for
$s\in S = \{a,a^{-1},b,b^{-1}\}$ (recall Subsection \ref{sec:gluing-graphs}).

Let $\mathcal{Z}_s$ denote the Schreier graph with vertices in $\Z$, edges $(x+1,x)$ labeled by $s$
and a self loop with label $\xi \not\in \{s,s^{-1} \}$ at each vertex.
Let $\varphi_s : \freegroup \to \Z$ be the homomorphism defined via $s \mapsto -1, s^{-1} \mapsto 1$
and $\xi \mapsto 0$ for $\xi \not\in \{s,s^{-1} \}$.
Then, the subgroup corresponding to $\mathcal{Z}_s$ is $\ker \varphi_s$.
Since $\freegroup / \ker \varphi_s$ is abelian, we have that $G_m \lhd 
[\freegroup , \freegroup] \lhd \ker \varphi_s$.

It is now immediate that $K_n$ is 
locally-$(\Lambda , \mathcal{Z}_a ,  \mathcal{Z}_{a^{-1} }, \mathcal{Z}_b , \mathcal{Z}_{b^{-1} } )$.
By Proposition \ref{prop:local}, this implies that 
$$ G_m \leq N \cap \bigcap_{s \in S} \ker \varphi_s \leq \core_\emptyset (K_n) . $$
Thus, $\core_\emptyset (K_n)$ is co-nilpotent.
\end{proof}

We are now ready to prove our main result, Theorem \ref{thm:main-free}.

\begin{proof}[Proof of Theorem \ref{thm:main-free}]
By Corollary \ref{cor:sequence-of-subgroups},
Proposition \ref{prop:growing fixing general} and Lemma \ref{lem:locally co-nilpotent}
it suffices to find $N \lhd \freegroup$ such that $\freegroup / N$ is nilpotent and 
transient.  
Since any $2$-generated nilpotent group is a quotient of $\freegroup$, it suffices 
to find a $2$-generated nilpotent group, whose Cayley graph is transient,
and then take $N$ to be the corresponding kernel of the canonical projection from $\freegroup$ onto our nilpotent group.  

It is well known that recurrent subgroups must have 
at most quadratic volume growth (this follows for instance from combining the 
Coulhon-Sallof-Coste inequality \cite{CSC}
and the method of evolving sets \cite{evolv_sets}, 
see {\em e.g.}\ \cite[Chapters 5 \& 8]{Gabor}).
Thus, it suffices to find a $2$-generated nilpotent group that has volume growth larger than quadratic.
This is not difficult, and there are many possibilities.

For a concrete example, we may consider take $\mathbb{H}_3(\Z)$, the Heisenberg group (over $\Z$).
This is the group whose elements are $3 \times 3$ matrices of the form
$\left[ \begin{smallmatrix} 
1 & a & c \\
0 & 1 & b \\
0 & 0 & 1
\end{smallmatrix} \right]$ with $a,b,c \in \Z$.
It is well known that $\mathbb{H}_3(\Z)$ is two-step nilpotent, with 
the commutator subgroup $[\mathbb{H}_3(\Z),\mathbb{H}_3(\Z)]$ containing matrices of the form 
$\left[\begin{smallmatrix} 
1 & 0 & c \\
0 & 1 & 0 \\
0 & 0 & 1
\end{smallmatrix}\right]$.
It follows that  $[ [\mathbb{H}_3(\Z),\mathbb{H}_3(\Z)],\mathbb{H}_3(\Z)] = \{1\}$ 
and $\mathbb{H}_3(\Z) / [\mathbb{H}_3(\Z),\mathbb{H}_3(\Z)] \cong \Z^2$.
This structure shows that $\mathbb{H}_3(\Z)$ has volume growth at least like a degree $3$ polynomial
(in fact the volume growth is like $r^4$), and thus must be transient.

Finally, $\mathbb{H}_3(\Z)$ is generated as a group by two elements $\left[\begin{smallmatrix} 
1 & 1 & 0 \\
0 & 1 & 0 \\
0 & 0 & 1
\end{smallmatrix}\right]$
and $\left[\begin{smallmatrix} 
1 & 0 & 0 \\
0 & 1 & 1 \\
0 & 0 & 1
\end{smallmatrix}\right]$, so can be seen as a quotient of $\freegroup$,
{\em i.e.}\ $\mathbb{H}_3(\Z) \cong \freegroup / N$ for some $N$.
This $N$ will satisfy our requirements, and this finishes the proof for $\freegroup$.

Let us note that 
for $\freegroup[r]$ with $r>2$ the construction is even simpler: 
Since $\Z^r$ is abelian and transient for any $r>2$, 
we can just take $N$ to be the commutator subgroup of $\freegroup[r]$. 
\end{proof}

\subsection{Proof of Theorem \ref{thm:free group no gap}}

\begin{proof}[Proof of Theorem \ref{thm:free group no gap}]
Let $N$ be a normal co-nilpotent subgroup in $\freegroup$, and let $\Lambda$ be
its associated Schreier graph. Choose some edge $\edge$ in $\Lambda$ labeled with $s \in S$ and
consider $(\Lambda,\edge)$ as a marked pair.
For any $n>0$ let $K_n \leq \freegroup$ be the subgroup corresponding to the root 
of the Shcreier graph $\Gamma_n(\Lambda,\edge)$, where 
$\Gamma_n(\Lambda,\edge)$ is defined in Section \ref{sec:gluing-graphs}.

We have already seen in the proof of Lemma  \ref{lem:locally co-nilpotent}
that $\core_\emptyset(K_n)$ is co-nilpotent in $\freegroup$, so the entropy function
is continuous and the interval of values 
$[0, c]$ is realizable for $c=\rweq{\nI[\s_n]}]$, 
by Corollary \ref{cor:single-subgroup}.

So we only need to prove that $\rweq{\nI[\s_n]} >0$ for some $n$.
(Of course we proved that when $\mu$ has finite support this converges to 
the random walk entropy, but here we are dealing with any finite entropy 
generating probability measure $\mu$, not necessarily with finite support.)

By the well known Kaimanovich-Vershik entropy criterion \cite{kaimanovich1983random}, 
if there exists a non-constant bounded harmonic function on the Schreier graph 
$\Gamma_n(\Lambda,\edge)$,
then the entropy $\rweq{\nI[\s_n]}$ is positive.

If the graph $(\Lambda,\edge)$ is transient, then 
considering the random walk on $\Gamma_n(\Lambda,\edge)$, 
for each glued copy of $(\Lambda,\edge)$ the random walk has 
positive probability to eventually end up in that copy.  
In other words,  there exists $0< \alpha < 1$ such that 
for any $|v| = n$ we have 
$$ \Pr [ \exists \ t_0 \ : \ \forall \ t > t_0  \  \tilde Z_t \in \shd(v) ] \geq \alpha . $$
Thus, for a fixed $|v|=n$, the function 
$$ h(g) = h_v(g) := \Pr_g [ \exists \ t_0 \ : \ \forall \ t > t_0  \  \tilde Z_t \in \shd(v) \ | \  \tilde Z_0 = g ] $$
is a non-constant bounded harmonic function on $G$.

Finding a transient Schreier graph $(\Lambda,\edge)$ with corresponding subgroup $N \lhd \freegroup$
which is co-nilpotent, can be done exactly as in the proof of Theorem  \ref{thm:main-free}.
\end{proof}

\bibliography{bib_iirss}

\begin{bibdiv}
\begin{biblist}

\bib{abert2012growth}{article}{
      author={Abert, Miklos},
      author={Bergeron, Nicolas},
      author={Biringer, Ian},
      author={Gelander, Tsachik},
      author={Nikolov, Nikolay},
      author={Raimbault, Jean},
      author={Samet, Iddo},
       title={On the growth of ${L}^2$-invariants for sequences of lattices in
  {L}ie groups},
        date={2012},
     journal={arXiv:1210.2961},
}

\bib{abert2011measurable}{article}{
      author={Abert, Miklos},
      author={Glasner, Yair},
      author={Virag, Balint},
       title={The measurable {K}esten theorem},
        date={2011},
     journal={arXiv:1111.2080},
}

\bib{abert2014kesten}{article}{
      author={Ab{\'e}rt, Mikl{\'o}s},
      author={Glasner, Yair},
      author={Vir{\'a}g, B{\'a}lint},
       title={{K}esten\'s theorem for invariant random subgroups},
        date={2014},
     journal={Duke Mathematical Journal},
      volume={163},
      number={3},
       pages={465\ndash 488},
}

\bib{bowen2010random}{article}{
      author={Bowen, Lewis},
       title={Random walks on random coset spaces with applications to
  {F}urstenberg entropy.},
        date={2014},
     journal={Inventiones mathematicae},
      volume={196},
      number={2},
}

\bib{bowen2015invariant}{article}{
      author={Bowen, Lewis},
       title={Invariant random subgroups of the free group},
        date={2015},
     journal={Groups, Geometry, and Dynamics},
      volume={9},
      number={3},
       pages={891\ndash 916},
}

\bib{bowen2012invariantlamp}{article}{
      author={Bowen, Lewis},
      author={Grigorchuk, Rostislav},
      author={Kravchenko, Rostyslav},
       title={Invariant random subgroups of lamplighter groups},
        date={2015},
     journal={Israel Journal of Mathematics},
      volume={207},
      number={2},
       pages={763\ndash 782},
}

\bib{bowen2014generic}{article}{
      author={Bowen, Lewis},
      author={Hartman, Yair},
      author={Tamuz, Omer},
       title={Generic stationary measures and actions},
        date={2017},
     journal={Transactions of the American Mathematical Society},
}

\bib{burton2016weak}{article}{
      author={Burton, Peter},
      author={Lupini, Martino},
      author={Tamuz, Omer},
       title={Weak equivalence of stationary actions and the entropy
  realization problem},
        date={2016},
     journal={arXiv preprint arXiv:1603.05013},
}

\bib{choquet1960lequation}{article}{
      author={Choquet, Gustave},
      author={Deny, Jacques},
       title={Sur lequation de convolution mu= mu-star-sigma},
        date={1960},
     journal={Comptes rendus hebdomadaires des seances de l'academie des
  sciences},
      volume={250},
      number={5},
       pages={799\ndash 801},
}

\bib{CSC}{article}{
      author={Coulhon, Thierry},
      author={Saloff-Coste, Laurent},
       title={Isop{\'e}rim{\'e}trie pour les groupes et les vari{\'e}t{\'e}s},
        date={1993},
     journal={Rev. Mat. Iberoamericana},
      volume={9},
      number={2},
       pages={293\ndash 314},
}

\bib{cover2012elements}{book}{
      author={Cover, Thomas~M},
      author={Thomas, Joy~A},
       title={Elements of information theory},
   publisher={John Wiley \& Sons},
        date={2012},
}

\bib{furman2002random}{article}{
      author={Furman, Alex},
       title={Random walks on groups and random transformations},
        date={2002},
     journal={Handbook of dynamical systems},
      volume={1},
       pages={931\ndash 1014},
}

\bib{furstenberg1971random}{article}{
      author={Furstenberg, H.},
       title={Random walks and discrete subgroups of {L}ie groups},
        date={1971},
     journal={Advances in Probability and Related Topics},
      volume={1},
       pages={1\ndash 63},
}

\bib{furstenberg1973boundary}{article}{
      author={Furstenberg, H},
       title={Boundary theory and stochastic processes on homogeneous spaces},
        date={1973},
     journal={Harmonic analysis on homogeneous spaces},
      volume={26},
       pages={193\ndash 229},
}

\bib{furstenberg2009stationary}{incollection}{
      author={Furstenberg, H.},
      author={Glasner, E.},
       title={Stationary dynamical systems},
        date={2010},
   booktitle={Dynamical numbers---interplay between dynamical systems and
  number theory},
      series={Contemp. Math.},
      volume={532},
   publisher={Amer. Math. Soc.},
     address={Providence, RI},
       pages={1\ndash 28},
      review={\MR{2762131}},
}

\bib{hartman2012furstenberg}{article}{
      author={Hartman, Yair},
      author={Tamuz, Omer},
       title={{F}urstenberg entropy realizations for virtually free groups and
  lamplighter groups},
        date={2015},
     journal={Journal d'Analyse Math{\'e}matique},
      volume={126},
      number={1},
       pages={227\ndash 257},
}

\bib{hartman2013stabilizer}{article}{
      author={Hartman, Yair},
      author={Tamuz, Omer},
       title={Stabilizer rigidity in irreducible group actions},
        date={2016},
     journal={Israel Journal of Mathematics},
      volume={216},
      number={2},
       pages={679\ndash 705},
}

\bib{kaimanovich2002poisson}{article}{
      author={Kaimanovich, Vadim~A},
       title={The poisson boundary of amenable extensions},
        date={2002},
     journal={Monatshefte f{\"u}r Mathematik},
      volume={136},
      number={1},
       pages={9\ndash 15},
}

\bib{kaimanovich2005amenability}{article}{
      author={Kaimanovich, Vadim~A},
       title={Amenability and the liouville property},
        date={2005},
     journal={Israel Journal of Mathematics},
      volume={149},
      number={1},
       pages={45\ndash 85},
}

\bib{kaimanovich1983random}{article}{
      author={Kaimanovich, Vadim~A},
      author={Vershik, Anatoly~M},
       title={Random walks on discrete groups: boundary and entropy},
        date={1983},
     journal={The annals of probability},
       pages={457\ndash 490},
}

\bib{leemann2016}{article}{
      author={Leemann, Paul-Henry},
       title={Schreir graphs: Transitivity and coverings},
        date={2016},
     journal={International Journal of Algebra and Computation},
      volume={26},
      number={01},
       pages={69\ndash 93},
}

\bib{LyonsPeres}{book}{
      author={Lyons, Russell},
      author={Peres, Yuval},
       title={Probability on trees and networks},
   publisher={Cambridge University Press},
        date={2016},
      volume={42},
}

\bib{evolv_sets}{article}{
      author={Morris, Ben},
      author={Peres, Yuval},
       title={Evolving sets, mixing and heat kernel bounds},
        date={2005},
     journal={Probability Theory and Related Fields},
      volume={133},
      number={2},
       pages={245\ndash 266},
}

\bib{nevo2003spectral}{article}{
      author={Nevo, Amos},
       title={The spectral theory of amenable actions and invariants of
  discrete groups},
        date={2003},
     journal={Geometriae Dedicata},
      volume={100},
      number={1},
       pages={187\ndash 218},
}

\bib{nevo2000rigidity}{inproceedings}{
      author={Nevo, Amos},
      author={Zimmer, Robert~J},
       title={Rigidity of {F}urstenberg entropy for semisimple {L}ie group
  actions},
        date={2000},
   booktitle={Annales scientifiques de l'ecole normale sup{\'e}rieure},
      volume={33},
       pages={321\ndash 343},
}

\bib{Gabor}{book}{
      author={Pete, Gabor},
       title={Probability and geometry on groups},
        note={available at: \url{http://math.bme.hu/~gabor/PGG.pdf}},
}

\bib{raugi2004general}{article}{
      author={Raugi, Albert},
       title={A general {C}hoquet--{D}eny theorem for nilpotent groups},
        date={2004},
     journal={Annales de l'{IHP} probabilit{\'e}s et statistiques},
      volume={40},
      number={6},
       pages={677\ndash 683},
}

\bib{vershik2012totally}{article}{
      author={Vershik, Anatolii~Moiseevich},
       title={Totally nonfree actions and the infinite symmetric group},
        date={2012},
     journal={Mosc. Math. J},
      volume={12},
      number={1},
       pages={193\ndash 212},
}

\end{biblist}
\end{bibdiv}

\end{document}